\documentclass[11pt]{amsart}
\usepackage[utf8]{inputenc}
\usepackage[a4paper,margin=3cm]{geometry}
\usepackage[pdftex,pdfborder={0 0 0},colorlinks]{hyperref}
\usepackage{latexsym,lmodern,graphicx,enumitem,amssymb}
\usepackage{dsfont,txfonts}
\usepackage{epsfig,psfrag}

\newtheorem{thm}[equation]{Theorem}
\newtheorem{lem}[equation]{Lemma}
\newtheorem{prop}[equation]{Proposition}
\newtheorem{cor}[equation]{Corollary}
\newtheorem{defi}[equation]{Definition}
\newtheorem{conj}[equation]{Conjecture}
\newtheorem{rem}[equation]{Remark}
\numberwithin{equation}{section}

\newcommand{\R}{\mathbb{R}}
\newcommand{\X}{\mathcal{X}}
\newcommand{\N}{\mathbb{N}}

\newcommand{\E}{\mathbb{E}}

\newcommand{\1}{\mathds{1}}

\begin{document}
\title[From dimension free concentration to the Poincar\'e inequality]
{From dimension free concentration to the Poincar\'e inequality}
\author{N. Gozlan, C. Roberto, P-M. Samson}

\date{\today}

\address{Universit\'e Paris Est Marne la Vall\'ee - Laboratoire d'Analyse et de Math\'e\-matiques Appliqu\'ees (UMR CNRS 8050), 5 bd Descartes, 77454 Marne la Vall\'ee Cedex 2, France}

\address{Universit\'e Paris Ouest Nanterre la D\'efense, MODAL'X, EA 3454, 200 avenue de la R\'epublique 92000 Nanterre, France}

\email{nathael.gozlan@univ-mlv.fr, croberto@math.cnrs.fr, paul-marie.samson@univ-mlv.fr}

\thanks{The authors were partially supported by the ``Agence Nationale de la Recherche'' through the grants ANR 2011 BS01 007 01 and  ANR 10 LABX-58.}


\begin{abstract}
We prove that a probability measure on an abstract metric space satisfies a non trivial dimension free concentration inequality for the $\ell_2$ metric if and only if it satisfies the Poincar\'e inequality.
Under some additional assumptions, our result extends to convex sets situation.
\end{abstract}
\maketitle

\section{Introduction}
Throughout the paper $(\X,d)$ denotes a polish metric space and $\mathcal{P}(\X)$ the set of Borel probability measures on  $\X.$
On the product space $\X^n$, we consider the following $\ell_p$ product distance $d_p$ defined by
\[
d_p(x,y) = \left[ \sum_{i=1}^n d^p(x_i,y_i)\right]^{1/p},\qquad x,y\in \X^n,\qquad p\geq 1.
\]
(Note that the dependence on the dimension $n$ is understood.)
If $A$ is a Borel subset of $\X^n$, we define its enlargement $A_{r,p}$ (simply denoted by $A_r$ when $n=1$), $r\geq0$ as follows
\[
A_{r,p} =\{ x \in \X^n ;  d_p(x,A) \leq r\}.
\]
Also, in all what follows, $\alpha : \R^+ \to \R^+$ will always be a non increasing function. One will say that $\mu \in \mathcal{P}(\X)$ satisfies the \emph{dimension free concentration property} with the concentration profile $\alpha$ and with respect to the $\ell_p$ product structure if
\begin{equation}\label{eq:CI}
\mu^n(A_{r,p})\geq 1-\alpha(r),\qquad \forall r\geq 0,
\end{equation}
for all $A\subset \X^n$, with $\mu^n(A)\geq 1/2$. For simplicity,  we will often say that $\mu$ satisfies the dimension free concentration inequality $\mathbf{CI}_p^\infty(\alpha)$, and, if $\mu$ satisfies \eqref{eq:CI} only for $n=1$, we will say that $\mu$ satisfies $\mathbf{CI}(\alpha).$
We refer to \cite{Ledoux-book} for an introduction on the notion of concentration of measure.

The general problem considered in this paper is to give a characterization of the class of probability measures satisfying $\mathbf{CI}_p^\infty(\alpha)$. Our main result shows that the class of probability measures satisfying $\mathbf{CI}^\infty_2(\alpha)$, for some non trivial $\alpha$, always contains the class of probability measures satisfying the Poincar\'e inequality. Moreover, these two classes coincide when $\alpha$ is exponential: $\alpha(r)=be^{-ar}$, for some $a,b>0.$

Before stating this result, let us recall the definition of the Poincar\'e inequality:  one says that $\mu \in \mathcal{P}(\X)$ satisfies the \emph{Poincar\'e inequality} with the constant $\lambda \in \R^+\cup\{+\infty\}$, if 
\begin{equation}\label{Poincare}
\lambda\mathrm{Var}_\mu (f) \leq \int |\nabla^- f|^2\,d\mu,
\end{equation}
for all Lipschitz function $f:\X\to \R$, where by definition 
\[
|\nabla^- f|(x) = \limsup_{y \to x} \frac{[f(y)-f(x)]_-}{d(y,x)}, 
\qquad \qquad (\mbox{with } [X]_-:=\max(-X,0))
\]
when $x$ is not isolated in $\X$ (we set $|\nabla^- f|(x)=0$, when $x$ is isolated in $\X$).
By convention $\infty \times 0 =0$, so that $\lambda=+\infty$ if and only if $\mu$ is a Dirac measure.

\begin{rem}\label{rem Poincare}Let us make a few comments about the definition of the Poincar\'e inequality.
\begin{enumerate}
\item Since the right hand side of \eqref{Poincare} is always finite when $f$ is Lipschitz, it follows in particular that Lipschitz functions always have finite variance when Poincar\'e inequality holds.
\item When $(\X,d)$ is a smooth Riemannian manifold equipped with its geodesic distance and $f:\X\to\R$, it is not difficult to check that if $f$ is differentiable at a point $x$, then $|\nabla^-f|(x)$ coincides with the norm of the \emph{vector} $\nabla f(x)$ (belonging to the tangent space at $x$). 
If $(\X,d)=(B,\|\,\cdot\,\|)$ is a Banach space, and $f:B\to \R$ is differentiable at $x\in B$, then $|\nabla^-f|(x)$ is equal to $\|Df(x)\|_*$, the dual norm of the differential $Df(x)\in B^*$ of $f$ at $x$.
So \eqref{Poincare} gives back the usual definitions in a smooth context.
\item To a Lipschitz function $f$ on $\X$, one can also associate $|\nabla^+ f|$ and $|\nabla f|$, which are defined by replacing $[\,\cdot\,]_-$ by $[\,\cdot\,]_+$ or by $|\,\cdot\,|$ respectively.
Since $|\nabla^+ f|=|\nabla^-(- f)|$, replacing $f$ by $-f$, we observe that the Poincar\'e inequality can be equivalently restated as 
\begin{equation}\label{Poincare+}
\lambda\mathrm{Var}_\mu (f) \leq \int |\nabla^+ f|^2\,d\mu,
\end{equation}
for all Lipschitz function $f:\X\to \R$. Moreover, since $|\nabla f|=\max(|\nabla^-f|;|\nabla^+f|)$, we see that \eqref{Poincare} and \eqref{Poincare+} both imply yet another version of Poincar\'e inequality (considered for instance in \cite{Ledoux-book} or \cite{BL97}):
\begin{equation}\label{Poincare abs}
\lambda\mathrm{Var}_\mu (f) \leq \int |\nabla f|^2\,d\mu,
\end{equation}
for all Lipschitz function $f:\X\to \R$. That \eqref{Poincare abs} also implies \eqref{Poincare} is not obvious. A proof of this fact can be found in \cite[Proposition 5.1]{GL13} (the result is stated for the logarithmic Sobolev inequality but the same conclusion holds for the Poincar\'e inequality). The proof relies on a technique, developed in \cite{AGS12}, consisting in relaxing the right hand side of \eqref{Poincare abs} and yielding to the notion of Cheeger energy.
\end{enumerate}
\end{rem}

%%%%%%%%%%%%%%%%%%%%%%%%%%%%%%%%%%%%%%%%%%%%%%%%%%%%%%%%%%%%%%%%%%%%%%%%%%%%%%%%%%%
%%%%%%%%%%%%%%%%%%%%%%%%%%%%%%%%%%%%%%%%%%%%%%%%%%%%%%%%%%%%%%%%%%%%%%%%%%%%%%%%%%%

\subsection{Main results.}

Denote by 
$\overline{\Phi}$ the tail distribution function of the standard Gaussian measure $\gamma(dx)=(2\pi)^{-1/2} e^{-x^2/2}\,dx$ on $\R$ defined by 
\[
\overline{\Phi} (x) = \frac{1}{\sqrt{2\pi}}\int_{x}^{+\infty} e^{-u^2/2}\,du,\qquad x\in \R.
\]
The main result of this paper is the following theorem. %In what follows, $\overline{\Phi}$ will denote the tail distribution function of the standard Gaussian measure $\gamma(dx)=(2\pi)^{-1/2} e^{-x^2/2}\,dx$ on $\R$ defined by 
%\[
%\overline{\Phi} (x) = \frac{1}{\sqrt{2\pi}}\int_{x}^{+\infty} e^{-u^2/2}\,du,\qquad x\in \R.
%\]

\begin{thm}\label{main result} 
If $\mu$ satisfies the dimension free concentration property $\mathbf{CI}_2^\infty(\alpha)$, then  $\mu$ satisfies the Poincar\'e inequality \eqref{Poincare} with the constant $\lambda$ defined by
\[
\sqrt{\lambda}= \sup\left\{\frac{ \overline{\Phi}^{-1} \left(\alpha(r)\right)}{r}; r> 0 \text{ s.t } \alpha(r)\leq 1/2  \right\}.
\]
Moreover, if $\alpha$ is convex decreasing and such that $\alpha(0)=1/2$, then $\lambda=(2\pi \alpha'_+(0)^2),$ where $\alpha'_+(0) \in [-\infty, 0)$ is the right derivative of $\alpha$ at $0.$
\end{thm}

Conversely, it is well known -- since the work by Gromov and Milman \cite{GM83} (see also \cite{AS94}, \cite{BL97}, \cite{Sch98}, \cite{Goz10} for related results) -- that a probability measure $\mu$ verifying the Poincar\'e inequality satisfies a dimension free concentration property with a profile of the form $\alpha(r)=be^{-ar}$, for some $a,b>0$. We recall this property in the following theorem and refer to Section \ref{sec:gromov-milman} for a proof. 

\begin{thm}[Gromov and Milman]\label{Gromov Milman}
Suppose that $\mu$ satisfies the Poincar\'e inequality \eqref{Poincare} with a constant $\lambda>0$, then it satisfies the dimension free concentration property with the profile 
\[
\alpha(r)=b\exp(-a\sqrt{\lambda}r),\qquad r\geq0,
\]
where $a,b$ are universal constants.
\end{thm}

Thus, Theorem \ref{main result} and Theorem \ref{Gromov Milman} give a full description of the set of probability distributions verifying a dimension free concentration property with a concentration profile $\alpha$ such that $\{r :\alpha(r)<1/2\}\neq \emptyset$ : this set coincides with the set of probability measures verifying the Poincar\'e inequality. An immediate corollary of Theorem \ref{main result} and Theorem \ref{Gromov Milman} (see Corollary \ref{Cor Tal} below) is that any type of dimension free concentration inequality can always be improved into a dimension free concentration inequality with an exponential profile (up to universal constants). This was already noticed by Talagrand in \cite{Tal91}. See Section \ref{self improvement} for a further discussion.

\begin{rem}\label{rem-intro}
Let us make some comments on the constant $\lambda$ appearing in Theorem \ref{main result}.
\begin{enumerate}
\item We observe first that $\lambda>0$ if and only if there is some $r_o>0$ such that $\alpha(r_o)<1/2.$ In particular, Theorem \ref{main result} applies in the case of the following ``minimal" profile $\alpha=\beta_{a_o,r_o}$, defined as follows
\begin{equation}\label{minimal profile}
\beta_{a_o,r_o}(r)=1/2,\ \text{ if } r<a_o\qquad\text{and}\qquad \beta_{a_o,r_o}(r)=a_o,\ \text { if } r\geq r_o,
\end{equation}
where  $a_o \in [0,1/2)$, $r_o>0$.
If a probability measure satisfies $\mathbf{CI}_2^\infty(\beta_{a_o,r_o})$, then it satisfies the Poincar\'e inequality  with the constant
\[
\sqrt{\lambda_{a_o,r_o}} :=  \frac{ \overline{\Phi}^{-1} \left(a_o\right)}{r_o}
\]
\item Then we notice that any non increasing function $\alpha:\R^+\to \R^+$, with $\alpha(0)=1/2$, can be written as an infimum of minimal profiles: 
\[
\alpha = \inf_{r>0} \beta_{\alpha(r),r}.
\]
Therefore, the constant $\lambda$ given in Theorem \ref{main result} is the supremum of the constants $\lambda_{\alpha(r),r}$ $r>0$ defined above. This shows that the information contained in the concentration profile $\alpha$ is treated pointwise, and that the global behavior of $\alpha$ is not taken into account. 
\item It is well known that the standard Gaussian measure $\gamma$ satisfies the dimension free concentration property with the profile $\alpha=\overline{\Phi}$ (this follows from the Isoperimetric theorem in Gauss space due to Sudakov-Cirelson \cite{SC74} and Borell \cite{Bor75}, see \textit{e.g.}\ \cite{Ledoux-book}). Hence, applying Theorem \ref{main result}, we conclude that $\gamma$ satisfies the Poincar\'e inequality with the constant $\lambda=1$, which is well known to be optimal (see \textit{e.g.}\ \cite[Chapter 1]{ane}).
\item Finally we observe that, if the concentration profile $\alpha(r)$ goes to zero too fast when $r\to \infty$, then $\lambda=+\infty$ and so $\mu$ is a Dirac measure. This happens for instance when $\alpha(r)=be^{-ar^k}$, $r\geq0$, with $k>2$ and $a,b>0.$
\end{enumerate}
\end{rem}
Theorem \ref{main result} completes a previous result obtained by the first author \cite{Goz09} (see also \cite{GRS11}), namely that the Gaussian dimension free concentration is characterized by a transport-entropy inequality. We now state this result and start by recalling some notation.
The Kantorovich-Rubinstein distance $W_p$, $p\geq1$, between $\nu,\mu \in \mathcal{P}(\X)$ is defined by
$$
W_p^p(\nu,\mu)=\inf_{\Pi(\nu,\mu)} \E[d^p(X,Y)],
$$
where the infimum runs over the set $\Pi(\nu,\mu)$ of couples of random variables $(X,Y)$ such that $\mathrm{Law}(X)=\mu$ and $\mathrm{Law}(Y)=\nu$. Then, a probability measure is said to satisfy the \emph{$p$-Talagrand transport-entropy inequality} with constant $C>0$, if it holds
\begin{equation}\label{p-Tal}
W_p^p(\nu,\mu) \leq CH(\nu|\mu),\qquad \forall \nu \in \mathcal{P}(\X),
\end{equation}
where the relative entropy functional is defined by $H(\nu|\mu)=\int \log \frac{d\nu}{d\mu}\,d\nu$ if $\nu$ is absolutely continuous with respect to $\mu$, and $H(\nu|\mu)=+\infty$ otherwise. Inequalities of this type were introduced by Marton and Talagrand in the nineties \cite{M86,Tal96}. We refer to the survey \cite{GL10} for more informations on this topic.
\begin{thm}[\cite{Goz09}] \label{thm:Goz09}
Fix $p\geq2$ and $C>0$. Then, a  probability measure $\mu$ satisfies the $p$-Talagrand transport inequality \eqref{p-Tal} with constant $C$ if and only if it satisfies the dimension free concentration inequality $\mathbf{CI}_p^\infty(\alpha)$, with a concentration profile of the form
\[
\alpha(r)=\exp\left(-\frac{1}{C} [r-r_o]_+^p\right),\qquad r\geq0,
\]
for some $r_o\geq0$ (with  $[X]_+:=\max(X,0)$).
\end{thm}
As we will see, the proofs of Theorem  \ref{main result} and \ref{thm:Goz09} are very different. Both make use of probability limit theorems, but not at the same scale: Theorem \ref{thm:Goz09} used Sanov's large deviations theorem, whereas Theorem \ref{main result} is an application of the central limit theorem. Moreover, contrary to what happens in Theorem \ref{main result} (see Item (2) of Remark \ref{rem-intro}), the global behavior of the concentration profile is used in Theorem \ref{thm:Goz09}.

In view of Theorems \ref{main result} and \ref{thm:Goz09}, it is natural to formulate the following general question:
\begin{itemize}
\item[$\mathbf{(Q)}$] Which functional inequality is equivalent to $\mathbf{CI}_p^\infty(\alpha)$ for a concentration profile of the form
\[
\alpha(r)=\exp(-a[r-r_o]^k_+),\qquad r\geq0,
\]where $a>0,r_o\geq0$ and $k>0$ ?
\end{itemize}

\begin{rem}
It is easy to see, using the Central Limit Theorem, that for $p\in [1,2)$ the only probability measures verifying $\mathbf{CI}_p^{\infty}(\alpha)$, for some $\alpha$ such that $\alpha(r_o)<1/2$ for at least one $r_o>0$, are Dirac masses. Thus the question $\mathbf{(Q)}$ is interesting only for $p\geq2.$
\end{rem}

To summarize, Theorem \ref{thm:Goz09} shows that the answer to $\mathbf{(Q)}$ is the $p$-Talagrand inequality for $k=p$ and $p\geq2$. Theorem \ref{main result} shows that the answer is the Poincar\'e inequality for $p=2$ and for $k\in (0,1]$. Moreover Point (4) of Remark \ref{rem-intro} above shows that for $p=2$, the question is interesting only for $k\in [1;2].$ The case $k\in (1;2)$ is still open.

Finally, we mention that some partial results are known for $p=\infty$. Indeed, in \cite{BH00}, Bobkov and Houdr\'e characterized the set of probability measures on $\R$ satisfying $\mathbf{CI}_\infty^\infty(\beta_{a_o,r_o})$, with $a_o \in [0,1/2)$, where $\beta_{a_o,r_o}$ is the minimal concentration profile defined by \eqref{minimal profile}. They showed that a probability measure $\mu$ belongs to this class if and only if the map $U_\mu$ defined by
\[
U_\mu(x)=F_\mu^{-1}\left(\frac{1}{1+e^{-x}}\right),\qquad x\in \R,
\]
where $F_\mu(x)=\mu((-\infty,x])$ and $F_\mu^{-1}(p)=\inf\{x\in \R ; F_\mu(x)\geq p\}$, $p\in (0,1)$, satisfies the following inequality on the interval where it is defined:
\[
|U_\mu(x)-U_\mu(y)|\leq a+b|x-y|,
\]
for some $a,b\geq0.$

Finally, we mention that, under additional assumptions, Theorem \ref{main result} extends to convex sets. 
More precisely, we will prove that, \textit{inter alia}, in geodesic spaces in which the distance function is convex (Busemann spaces),  a probability measure satisfies $\mathbf{CI}^\infty_2(\alpha)$ restricted to \emph{convex} sets, with a non trivial profile $\alpha$, if and only if it satisfies the Poincar\'e inequality restricted to \emph{convex} functions (see Section 6 for a precise statement).
This generalizes a previous result by Bobkov and G\"otze that was only valid for probability on the real line \cite[Theorem 4.2]{BG99}.

\subsection{Alternative formulation in terms of observable diameters} It is possible to give an alternative formulation of Theorem \ref{main result} and Theorem \ref{Gromov Milman} using the notion of observable diameter introduced by Gromov (\cite[Chapter 3.1/2]{Gromov-book}).
Recall that, if $(\X,d,\mu)$ is a metric space equipped with a probability measure and $t \in [0,1]$, the partial diameter of $(\X,d)$ is defined as the infimum of the diameters of all the subsets $A\subset \X$ satisfying $\mu(A)\geq 1-t.$ It is denoted by $\mathrm{Part\,Diam}(\X,d,\mu,t).$ If $f:\X\to\R$ is some $1$-Lipschitz function, we denote by $\mu_f \in \mathcal{P}(\R)$ the push forward of $\mu$ under $f$. Then, the \emph{observable diameter} of $(\X,d,\mu)$ is defined as follows
\[
\mathrm{Obs\,Diam}(\X,d,\mu,t)=\sup_{f\ 1-\text{Lip}} \mathrm{Part\,Diam}(\R,|\,\cdot\,|,\mu_f,t)\in \R^+\cup\{+\infty\} .
\]
We define accordingly the observable diameters of $(\X^n,d_2,\mu^n,t)$ for all $n\in \N^*.$

The observable diameters are related to concentration profiles by the following lemma (see \textit{e.g.}\ \cite[Lemma 2.22]{FS12}).
\begin{lem}\label{lem:FS12}If $\mu$ satisfies $\mathbf{CI}(\alpha)$, then 
\[
\mathrm{Obs\,Diam}(\X,d,\mu,2\alpha(r)) \leq 2r,
\]
for all $r\geq0$ such that $\alpha(r)\leq 1/2.$\\
Conversely, for all $t\in [0,1/2]$, for all $A\subset \X$, with $\mu(A)\geq1/2$, it holds
\[
\mu(A_{r(t)}) \geq 1-t
\]
with $r(t)=\mathrm{Obs\,Diam}(\X,d,\mu,t).$
\end{lem}
The following corollary gives an interpretation of the Poincar\'e inequality in terms of the boundedness of the observable diameters of the sequence of metric probability spaces $(\X^n,d_2,\mu^n)_{n\in \N^*}.$

\begin{cor}\label{cor:obs diam} A probability measure $\mu$ on $(\X,d)$ satisfies the Poincar\'e inequality \eqref{Poincare} with the optimal constant $\lambda$ if and only if for some $t\in (0,1/2)$ $$r_\infty(t):=\sup_{n\in \N^*} \mathrm{Obs\,Diam}(\X^n,d_2,\mu^n,t)<\infty.$$
Moreover,
\[
\overline{\Phi}^{-1}(t)\leq r_\infty(t)\sqrt{\lambda} \leq a\log\left(\frac{b}{t}\right),\qquad \forall t\in (0,1/2)
\]
where $a>0$ and $b\geq1$ are some universal constants.

\end{cor}

%%%%%%%%%%%%%%%%%%%%%%%%%%%%%%%%%%%%%%%%%%%%%%%%%%%%%%%%%%%%%%%%%%%%%%%%%

\subsection{Tools.}
In this section, we briefly introduce the main tools that will be used in the proof of Theorem \ref{main result}: inf-convolution operators, related to both concentration and to Hamilton-Jacobi equations, and the Central Limit Theorem.

The first main tool in the proof of Theorem \ref{main result} is a new alternative formulation of concentration of measure  $\mathbf{CI}_p^\infty(\alpha)$   in terms of deviation inequalities for inf-convolution operators that was introduced in \cite{GRS11}. 
Recall that for all $t>0$, the infimum convolution operator $f\mapsto Q_t f$ is defined for all $f:\X^n \to \R\cup\{+\infty\}$ bounded from below as follows
\begin{equation}\label{operator}
Q_tf(x)=\inf_{y\in \X^n}\left\{ f(y) + \frac{1}{t^{p-1}} d^p_p(x,y)\right\},\qquad x\in \X^n
\end{equation}
(we should write $Q_t^{p,(n)}$, but we will omit, for simplicity, the superscripts $p$ and $(n)$ in the notation).

In the next proposition, we recall a result from \cite{GRS11}  that gives a new way to express concentration of measure (our first main tool).
\begin{prop}\label{prop-GRS11}
Let $\mu \in \mathcal{P}(\X)$; $\mu$ satisfies $\mathbf{CI}_p^\infty(\alpha)$ if and only if for all $n\in \N^*$ and for all measurable function $f:\X^n \to \R\cup\{+\infty\}$ bounded from below and such that $\mu^n(f=+\infty)<1/2$, it holds
\begin{equation}\label{eq-GRS11}
 \mu^n(Q_t f > m(f) +r)\leq \alpha(r^{1/p}t^{1-1/p}),\quad \forall r,t>0,
 \end{equation}
where $m(f)=\inf\{m\in\R;\mu^n(f\leq m) \geq 1/2\}.$
\end{prop}

The second main tool is the well known fact that the function $u:(t,x) \mapsto Q_tf(x)$ is, in some weak sense,  solution of the Hamilton-Jacobi equation 
\[
\frac{\partial u}{\partial t} = -\frac{1}{4} |\nabla_x u|^2.
\]
This result is very classical on $\R^k$ (see e.g \cite{Evans}); extensions to metric spaces were proposed in \cite{LV07}, \cite{Ba12}, \cite{AGS12} or \cite{GRS12}. This will be discussed in the next section.

The third tool is the Central Limit Theorem for triangular arrays of independent random variables  (see \textit{e.g.}\ \cite[p. 530]{F71}).
\begin{thm}\label{TCL}  For each $n$, let $X_{1,n}, X_{2,n},\ldots ,X_{n,n}$ be independent real random variables and define $T_n= X_{1,n}+\ldots +X_{n,n}.$ Assume that 
$\E[T_n^2] =1$,  $\E[X_{k,n}]=0$ for all $k\in\{0,\ldots,n\}$, and that the following Lindeberg condition holds, for all $t>0$
\begin{eqnarray}\label{lindeberg}
\sum_{k=1}^n \E \left[X_{k,n}^2\1_{|X_{k,n}|>t}\right]\to 0\quad \hbox{ as } n\to +\infty.
\end{eqnarray}
Then the distribution of $T_n$ converges weakly to the standard normal law.
% or equivalently the Levy distance $L(F_n,\Phi)$, between the law of $T_n$ and the standard normal law, tends to 0:
%\[
%L(F_n,\Phi):=\inf\left\{\varepsilon>0,  F_{n}(x -\varepsilon)-\varepsilon\leq \Phi(x)\leq F_{n}(x+\varepsilon)+\varepsilon \hbox{ for all }x\in \R\right\} \underset{n\to+\infty}{\longrightarrow} 0 
%\]
%where $F_{n}(x)=\P(T_n\leq x)$ and  $\Phi(x)=1-\overline \Phi(x), x\in \R$.
\end{thm}

%The third tool is the celebrated Berry-Esseen Inequality.

%\begin{thm}[Berry-Esseen]\label{BE}
%Let $(X_i)_{i\in \N^*}$ be an i.i.d. sequence of real random variables such that $\E[X_i]=0$, $\E[X_i^2] = \sigma^2$ and $\E[|X_i|^3]=\rho <\infty$. There exists a universal constant $\kappa>0$ such that, for all $n\in \N^*$,
%\[
%\sup_{x\in \R}\left|\P \left( \frac{X_1+\cdots+X_n}{\sqrt{n}\sigma} >x\right) - \overline{\Phi} (x)\right| \leq \kappa\frac{\rho}{\sigma^3 \sqrt{n}},
%\]
%where $\overline{\Phi}(x)=\frac{1}{\sqrt{2\pi}} \int_x^{+\infty} e^{-u^2/2}\,du,$ $x\in \R.$
%\end{thm}

We end this introduction with a short roadmap of the paper.
In Section 2 we make some comments on Theorem \ref{main result}. In particular, we compare Theorem \ref{main result} to a result by E. Milman on the Poincar\'e inequality in spaces with non-negative curvature and show, as an immediate consequence of our main result as well as E. Milman's result, that the celebrated KLS conjecture for isotropic log-concave probability measures can be reduced to some universal concentration inequalities (for isotropic
log-concave probability measures). In Section 3, we recall some properties of the infimum convolution operators that will be used in the proofs. Section 4 is dedicated to the proof of Theorem \ref{main result}
and Section 5 to the proof of Theorem \ref{Gromov Milman}, by means of the Herbst argument (the latter is somehow classical (see \textit{e.g.}\ \cite{ane,Ledoux-book}) but requires some care due to our general framework).  Finally Section 6 deals with the case of convex sets and the Poincar\'e inequality restricted to convex functions.

\medskip

\textbf{Acknowledgements.} The authors would like to thank Emanuel Milman for commenting the main result of this paper and for mentioning to them that the method of proof used by Talagrand in \cite{Tal91}, to prove Corollary \ref{Cor Tal}, could be extended to cover general situations.

%%%%%%%%%%%%%%%%%%%%%%%%%%%%%%%%%%%%%%%%%%%%%%%%%%%%%%%%%%%%%%%%%%%%%%%%
%%%%%%%%%%%%%%%%%%%%%%%%%%%%%%%%%%%%%%%%%%%%%%%%%%%%%%%%%%%%%%%%%%%%%%%%

\section{Comparison with other results}

In this section we collect some remarks and consequences of our main theorem.
First we shall compare our result to one of E. Milman, in Riemannian setting. Then, we state the celebrated KLS conjecture and give an equivalent formulation in terms of dimension free concentration property. Finally,
 we may comment on other type of dimension free concentration properties involving a different definition of enlargement.

\subsection{Dimension free concentration v.s non negative curvature.}
In Riemannian setting, Theorem \ref{main result} is reminiscent of the following recent result by E. Milman showing that under non-negative curvature the Poincar\'e constant of a probability measure can be expressed through very weak concentration properties of the measure \cite{Mil09a,Mil10}. 

We recall that the Minkowski content of a set $A \subset \X$ is defined as follows
\[
\mu^+(A) := \liminf_{r\to 0} \frac{\mu(A_r) - \mu(A)}{r}.
\]

\begin{thm}[Milman \cite{Mil10}]\label{Milman} Let $\mu(dx)=e^{-V(x)}\,dx$ be an absolutely continuous probability measure on a smooth complete separable Riemannian manifold $M$ equipped with its geodesic distance $d.$
Suppose that $V:M\to\R$ is a function of class $\mathcal{C}^2$ such that 
\[
\mathrm{Ric} + \mathrm{Hess}\,V \geq0,
\]
and that $\mu$ satisfies the following concentration of measure inequality
\[
\mu(A_r) \geq 1 - \alpha(r),\qquad \forall r\geq 0,
\]
with $\alpha:[0,\infty) \to [0,1/2]$ such that $\alpha(r_o)<1/2$, for some $r_o>0$. Then $\mu$ satisfies the following  Cheeger's inequality
\[
\mu^+(A)\geq D \min(\mu(A) ; 1-\mu(A)),\qquad \forall A \subset M,
\]
with 
\[
D=\sup\left\{ \frac{\Psi(\alpha(r))}{r} ; r>0 \text{ s.t } \alpha(r)<1/2\right\},
\]
where $\Psi:[0,1/2)$ is some universal function. 
\end{thm}
We recall that Cheeger's inequality with the constant $D$ implies the Poincar\'e inequality \eqref{Poincare} with the constant $\lambda=D^2/4$ \cite{Cheeger,Mazya}. In our result the non-negative curvature assumption of Milman's result is replaced by the assumption that the concentration is dimension free. 
\begin{rem}
Notice that, if $M$ has non-negative Ricci curvature and $\mu (dx) = \frac{1}{|K|} \mathbf{1}_K(x)\,dx$ is the normalized restriction of the Riemmanian volume to a  geodesically convex set $K$, then E. Milman also obtains \cite{Mil11} that 
\[
D= \sup\left\{\frac{1-2\alpha(r)}{r} ; r>0\right\}.
\]
This bound is optimal (see \cite{Mil11}).
\end{rem}

\subsection{A remark on the KLS conjecture}
In this section, $\mathcal{X}=\R^k$ is always equipped with its standard Euclidean norm $|\,\cdot\,|$.

Let us recall the celebrated conjecture by Kannan, Lov\'asz and Simonovits \cite{KLS95}.
Recall first that a probability measure $\mu$ on $\R^k$ is isotropic if $\int x\,\mu(dx)=0$ and $\int x_ix_j \,\mu(dx) =\delta_{ij}$ for all $1\leq i,j\leq n.$ It is log-concave if it has a density of the form $e^{-V}$, where $V:\R^k\to\R\cup\{+\infty\}$ is a convex function.

\begin{conj}[Kannan-Lov\'asz-Simonovits \cite{KLS95}]\label{KLS}
There is a universal constant $D>0$ such that for all $k\in \N^*$, any log-concave and isotropic probability measure $\mu$ on $\R^k$ satisfies the following Cheeger inequality 
\[
\mu^+(A)\geq D\min(\mu(A) ; 1-\mu(A)),\qquad \forall A \subset \R^k.
\]
Equivalently, there is a universal constant $\lambda>0$ such that for all $k\in \N^*$, any log-concave and isotropic probability measure $\mu$ on $\R^k$ satisfies the following Poincar\'e inequality
\[
\lambda\mathrm{Var}_\mu(f)\leq  \int |\nabla f|^2\,d\mu,
\]
for all $f:\R^k \to \R$ Lipschitz.
\end{conj}

According to E. Milman's Theorem \ref{Milman}, the above conjecture can be reduced to a statement about universal concentration inequalities for log-concave isotropic probabilities.

\begin{cor}\label{red KLS}
The KLS conjecture is equivalent to the following statement. There exists $r_o>0, a_o \in [0,1/2)$ such that for any $m\in \N^*$, any log-concave and isotropic probability $\nu$ on $\R^m$ satisfies 
\begin{equation}\label{KLS Milman}
\nu(A + r_o B_2) \geq 1-a_o,\qquad \forall A \subset \R^m \text{ s.t. } \nu(A)\geq1/2,
\end{equation}
where $B_2$ is the Euclidean unit ball of $\R^m.$
\end{cor}

This corollary follows immediately from Theorem \ref{Milman}. Below, we propose an alternative proof based on our main result (Theorem \ref{main result}).

\proof[Proof of Corollary \ref{red KLS}]According to Theorem \ref{Gromov Milman}, it is clear that the KLS conjecture implies uniform exponential concentration estimates for isotropic log-concave probability measures.

Conversely, let $\mu$ be isotropic and log-concave on $\R^k$. For all $n\in \N^*$, the probability $\mu^n$ is still isotropic and log-concave on $\left(\R^k\right)^n$. So applying \eqref{KLS Milman} to $\nu=\mu^n$ on $\left(\R^k\right)^n$, for all $n\in \N^*$, we conclude that $\mu$ satisfies $\mathbf{CI}_2^\infty(\beta_{a_o,r_o})$, where the concentration profile $\beta_{a_o,r_o}$ is defined by \eqref{minimal profile}. According to Theorem \ref{main result}, we conclude that $\mu$ satisfies Poincar\'e inequality with the constant $\lambda= \left(\overline{\Phi}^{-1}(a_o)/r_o\right)^2$. Since this holds for any isotropic log-concave probability measure in any dimension, this ends the proof.
\endproof

\subsection{Euclidean v.s Talagrand type enlargements}
Theorem \ref{main result} improves a preceding result by the first author \cite{Goz09} where a stronger form of exponential dimension free concentration, introduced by Talagrand \cite{Tal91,Tal95}, was shown to be equivalent to a transport-entropy inequality which was known to be also equivalent to the Poincar\'e inequality.
These equivalences, together with our main result, will allow us to prove that the two
notions of exponential dimension free concentration properties are actually equivalent (see Theorem \ref{canicule} below).

In order to present these equivalences, we need some notation and definition.

Given  $n\in \N^*$ and $A \subset \X^n$, consider the following family of enlargements of $A$: 
\[
\widetilde{A}_{a,r} =\left\{x \in \X^n ; \exists y\in A \text{ s.t } \sum_{i=1}^n \theta(ad(x_i,y_i))\leq r\right\},\qquad a >0,\quad r\geq0
\]
where $\theta(t) = t^2,$ if $t\in [0,1]$ and $\theta(t)=2t-1$, if $t\geq1.$

In the next definition, we recall the dimension free concentration property introduced by Talagrand.

\begin{defi}
A probability measure $\mu$ on $\X$ is said to satisfy the Talagrand exponential type dimension free concentration inequality with constants $a,b\geq0$ if for all $n\in \N^*$, for all $A\subset \X^n$ with $\mu^n(A)\geq1/2$, it holds
\begin{equation}\label{concentration Talagrand}
\mu^n(\widetilde{A}_{a,r})\geq 1-be^{-r},\qquad \forall r\geq0.
\end{equation}
\end{defi}

\begin{rem} \label{froid}
Using elementary algebra, one can compare the Talagrand concentration inequality \eqref{concentration Talagrand}
with the dimension free concentration inequality \eqref{eq:CI} under investigation in this paper.
More precisely we may prove that the former is stronger than the latter.
Indeed, since $t\mapsto \theta(\sqrt{t})$ is concave and vanishes at $0$, it is thus sub-additive. In turn, the following inequality holds
\[
\sum_{i=1}^n\theta(ad(x_i,y_i)) \geq \theta\left(\sqrt{\sum_{i=1}^n a^2d^2(x_i,y_i)}\right)=\theta(ad_2(x,y)),\qquad \forall x,y\in \X^n.
\]
Therefore,
\[
\widetilde{A}_{a,\theta(ar)} \subset A_{r,2},
\]
and so, if $\mu$ satisfies the Talagrand concentration inequality \eqref{concentration Talagrand}, then it obviously verifies the dimension free concentration inequality with the profile $\alpha(u)=be^{-\theta(au)}\leq ebe^{-2au},$ $u\geq0.$
\end{rem}

The following theorem summarizes the known links between Talagrand exponential type dimension free concentration and the Poincar\'e inequality.
\begin{thm}\label{Gozlan 09} 
Let $\mu$ be a probability measure on $\X$. The following statements are equivalent 
\begin{enumerate}
\item $\mu$ satisfies the Poincar\'e inequality \eqref{Poincare} with a constant $\lambda>0$.
\item $\mu$ satisfies the Talagrand exponential type dimension free concentration inequality \eqref{concentration Talagrand} with constants $a,b>0$.
%: for all $n\in \N^*$, for all $A \subset \X^n$ such that $\mu^n(A)\geq1/2$,
%\[
% \mu^n \left( \widetilde{A}_{a,r}\right) \geq 1 - be^{-r},\qquad \forall r\geq0,
% \]
%for some constants $a,b>0$.

\item  $\mu$ satisfies the following transport-entropy inequality
\[
\inf_{(X,Y) \in \Pi(\mu,\nu)}\mathbb{E} \left( \theta(Cd(X,Y)) \right) \leq  H(\nu|\mu),\qquad \forall \nu \in \mathcal{P}(\X),
\]
for some constant $C>0$, (recall that $\Pi(\mu,\nu)$ and the relative entropy $H(\nu|\mu)$ are defined before Theorem \ref{thm:Goz09}).
%is the set of couplings $\pi$ between $\mu$ and $\nu$ and the relative entropy is defined as follows $H(\nu|\mu)=\int \log \frac{d\nu}{d\mu}\,d\nu$, when $\nu \ll \mu$ and $+\infty$ otherwise.
\end{enumerate}
Moreover the constants above are related as follows :\\
$(1)\Rightarrow (2)$ with $a=\kappa \sqrt{\lambda}$ and $b=1$, for some universal constant $\kappa.$\\
$(2)\Rightarrow (3)$ with $C=a.$\\
$(3)\Rightarrow (1)$ with $\lambda=2C^2.$
\end{thm}
Let us make some bibliographical comments about the different implications in Theorem \ref{Gozlan 09}. The implication $(1)\Rightarrow (2)$ is due to Bobkov and Ledoux \cite{BL97}, the implication $(2)\Rightarrow (3)$ is due to the first author \cite[Theorem 5.1]{Goz09}, and the implication $(3)\Rightarrow (1)$ is due to Maurey \cite{Ma91} or Otto-Villani \cite{OV00}. The equivalence between (1) and (3) was first proved by Bobkov, Gentil and Ledoux in \cite{BGL01}.

\begin{rem}
It is worth noting that the implication $(2)\Rightarrow (3)$ follows from Theorem \ref{thm:Goz09} for $p=2$ by a change of metric argument. Namely, suppose that $\mu$ satisfies the concentration property (2) of Theorem \ref{Gozlan 09}, for some $a>0$, and define $\tilde{d}(x,y)=\sqrt{\theta(ad(x,y))}$ for all $x,y \in \X.$ It is not difficult to check that the function $\theta^{1/2}$ is subadditive, and therefore $\tilde{d}$ defines a new distance on $\X$. The $\ell_2$ extension of $\tilde{d}$ to the product $\X^n$ is
\[
\tilde{d}_2(x,y)= \left[\sum_{i=1}^n \theta(ad(x_i,y_i))\right]^{1/2},\qquad x,y\in \X^n,
\]
and it holds 
\[
\widetilde{A}_{a,r} = \{ x\in \X^n ; \tilde{d}_2(x,A) \leq \sqrt{r}\},\qquad \forall A \subset \X^n.
\]
Therefore, statement (2) can be restated by saying that $\mu$ satisfies $\mathbf{CI}_2^\infty (\alpha)$ (with respect to the distance $\tilde{d}$) with the Gaussian concentration profile $\alpha(r)=be^{-r^2}.$ Applying Theorem \ref{thm:Goz09}, we conclude that $\mu$ satisfies the $2$-Talagrand transport entropy inequality with the constant $1$ with respect to the distance $\tilde{d}$, which is exactly (3) with $C=a$.
\end{rem}

An immediate consequence of Theorem \ref{main result} and of Bobkov-Ledoux Theorem $(1)\Rightarrow (2)$ above is the following result showing the equivalence between the two forms of dimension free exponential concentration.

\begin{thm} \label{canicule}
Let $\mu$ be a probability measure on $\X$. The following are equivalent.
\begin{enumerate}
\item The probability measure $\mu$ satisfies the Talagrand exponential type dimension free concentration inequality \eqref{concentration Talagrand} with constants $a$ and $b$.
\item The probability measure $\mu$ satisfies $\mathbf{CI}_2^\infty (\alpha)$ with a profile $\alpha(u)=b'e^{-a'u},$ $u\geq0$.
\end{enumerate}
Moreover, the constants are related as follows:
$(1)\Rightarrow (2)$ with $a'=2a$ and $b'=eb$, and $(2)\Rightarrow (1)$ with $a=\kappa a'/\sqrt{\log(2b')}$ (for some universal constant $\kappa$) and $b=1$.
\end{thm}

We do not know if there exists a direct proof of the implication $(2) \Rightarrow (1)$.

\proof
We have already proved that (1) implies (2) (see Remark \ref{froid}). Let us prove the converse. According to Theorem \ref{main result} we conclude from (2) that $\mu$ satisfies the Poincar\'e inequality with a constant $\lambda \geq \left(\frac{\overline{\Phi}^{-1}(\alpha(u))}{u}\right)^{2},$ for all $u$ such that $\alpha(u)<1/2.$ A classical inequality gives $\overline{\Phi}(t)\leq \frac{1}{2}e^{-t^2/2},$ $t\geq0.$ Therefore, 
$\overline{\Phi}^{-1}(t)\geq 2\sqrt{-\log(2t)},$ for all $t\in (0,1/2)$. Hence, taking $u=2\log(2b')/a'$
(which guarantees that $\alpha(u)=1/(4b')$) yields to $\lambda \geq \frac{a'^2}{\log(2b')}$. According to the implication (1) $\Rightarrow$ (2) in Theorem \ref{Gozlan 09} we conclude that $\mu$ satisfies Talagrand concentration inequality \eqref{concentration Talagrand} with $a=\kappa a'/\sqrt{\log(2b')}.$
\endproof

%%%%%%%%%%%%%%%%%%%%%%%%%%%%%%%%%%%%%%%%%%%%%%%%%%%%%%%%%%%%%%%%%%%%%%%%%%%%%%%%%
%%%%%%%%%%%%%%%%%%%%%%%%%%%%%%%%%%%%%%%%%%%%%%%%%%%%%%%%%%%%%%%%%%%%%%%%%%%%%%%%%

\section{Some properties of the inf-convolution operators}

In this short (technical) section, we recall some properties of the inf-convolution operators related to Hamilton-Jacobi equations and to the concentration of measure, in the setting of metric spaces (recall that $(\X,d)$ is a complete separable metric space).

\subsection{Inf-convolution operators and Hamilton-Jacobi equations}

In this paragraph, we shall only consider  the case $p=2$. We do this restriction for simplicity and also since only that particular case will be used in the proof of Theorem \ref{main result} (in the next section). However we mention that most of the results of this section can be extended to any $p>1$. 

The following proposition collects two basic observations about the operators $Q_t$, $t>0.$  
\begin{prop}\label{prop:Qt} Let $h:\X\to \R$ be a Lipschitz function. Then
\begin{itemize}
\item[(i)] for all $x\in \X$, $Q_th(x) \to h(x)$, when $t\to 0^+$;
\item[(ii)] for all $\nu\in \mathcal{P}(\X)$, 
\begin{equation}\label{eq:HJ int}
\limsup_{t \to 0^+} \frac{1}{t}\int h(x)-Q_th(x)\,\nu(dx) \leq \frac{1}{4}\int |\nabla^-h(x)|^2\,\nu(dx).\end{equation}
\end{itemize}
\end{prop}

Before giving the proof of Proposition \ref{prop:Qt}, let us complete the picture by recalling the following theorem of \cite{AGS12} and \cite{GRS12bis} (improving preceding results of \cite{LV07} and \cite{Ba12}). This result will not be used in the sequel. 
\begin{thm} Let $h$ be a bounded function on a polish metric space $\X$. Then $(t,x)\mapsto Q_th(x)$ satisfies the following Hamilton-Jacobi (in)equation
\begin{equation}\label{HJ}
\frac{d}{dt_+} Q_th(x) \leq -\frac{1}{4}|\nabla Q_th|^2(x), \qquad \forall t>0,\ \forall x\in \X
\end{equation}
where $d/dt_+$ stands for the right derivative, and $|\nabla h|(x)=\limsup_{y\to x}\frac{|h(y)-h(x)|}{d(y,x)}.$
Moreover, if the space $\X$ is geodesic (\textit{i.e.}\ for all $x,y\in \X$ there exists at least one curve $(z_t)_{t\in [0,1]}$ such that $z_0=x$, $z_1=y$ and $d(z_s,z_t)=|t-s|d(x,y)$) then \eqref{HJ} holds with equality.
\end{thm}

Observe that, strangely, the two inequalities \eqref{eq:HJ int} and \eqref{HJ} goes in the opposite direction.
This suggests that, at $t=0$, there should be equality.

\proof[Proof of Proposition \ref{prop:Qt}]
\noindent
Let $L>0$ be a Lipschitz constant of $h$; since $Q_t h \leq h$ one has 
$$Q_th(x)=\inf_{y\in B(x, 2Lt)} \left\{h(y)+\frac{1}{t}d^2(x,y)\right\}.$$
(Namely, if $d(x,y)>2Lt$, it holds $h(y)-h(x)+\frac{1}{t}d^2(x,y)\geq \left(-L+\frac{1}{t}d(x,y)\right)d(x,y)>0$.)

Hence
\begin{align*}
0\leq\frac{h(x)-Q_th(x)}{t}&=\sup_{y\in B(x,2Lt)}\left\{\frac{h(x)-h(y)}{t} - \frac{d^2(x,y)}{t^2}\right\}\\
&\leq \sup_{y\in B(x,2Lt)}\left\{ \frac{[h(x)-h(y)]_+}{d(x,y)}\frac{d(x,y)}{t} - \frac{d^2(x,y)}{t^2}\right\}\\
&\leq \sup_{r \in \mathbb{R}} \left\{ \sup_{y\in B(x,2Lt)} \frac{[h(x)-h(y)]_+}{d(x,y)}r - r^2\right\}
=  \frac{1}{4}\sup_{y\in B(x,2Lt)} \frac{[h(x)-h(y)]_+^2}{d^2(x,y)}.
\end{align*}
We conclude from this that $0\leq (h-Q_th)/t\leq L^2/4$. This implies in particular that $Q_th\to h$ when $t\to0.$ Taking the $\limsup$ when $t\to0^+$ gives 
\begin{equation}\label{eq:HJ facile}
\limsup_{t\to 0^+}\frac{h(x)-Q_th(x)}{t} \leq \frac{1}{4}|\nabla^-h(x)|^2.
\end{equation} 
Inequality \eqref{eq:HJ int} follows from \eqref{eq:HJ facile} using Fatou's Lemma in its $\limsup$ version. The application of Fatou's Lemma is justified by the fact that the family of functions $\{(h-Q_th)/t\}_{t>0}$ is uniformly bounded. 
\endproof
\begin{rem}The proof of \eqref{eq:HJ facile} can also be found in \cite[Theorem 22.46]{Vi09}  (see also \cite[Proposition A.3]{GRS12bis}, \cite{LV07,Ba12,AGS12}).
\end{rem}

\subsection{Inf-convolution operators and concentration of measure}\label{link}

In this subsection we briefly recall the short proof of Proposition \ref{prop-GRS11} for the sake of completeness
 (Proposition that relates the infimum convolution operator and the concentration property $\mathbf{CI}_p^\infty(\alpha)$).

\proof Let $f:\X^n \to \R\cup\{+\infty\}$ be a function bounded from below with   $\mu^n(f=+\infty)<1/2$. By definition of $m(f)$, it holds $\mu^n(f\leq m(f)) \geq 1/2.$ Define $A=\{f \leq m(f)\}$. 
If $\mu$ satisfies the dimension free concentration property $\mathbf{CI}_p^\infty(\alpha)$, since by definition of $m(f)$, $\mu^n(A)\geq1/2$, one has  $\mu^n(\X^n \setminus A_{u,p}) \leq \alpha(u)$, for all $u\geq0.$ Then, observe that 
\[
Q_tf(x) \leq m(f) + \frac{1}{t^{p-1}}d^p_p(x,A),\qquad \forall x\in \X^n.
\]
Hence $\{Q_tf>m(f)+r\}\subset \{d_p(\,\cdot\,,A)>{r^{1/p}t^{1-1/p}}\} = \X\setminus A_{r^{1/p}t^{1-1/p},p}$, which proves \eqref{eq-GRS11}.

To prove the converse, take a Borel set $A \subset \X^n$ such that $\mu^n(A)\geq 1/2$ and consider the function $i_A$ equals to $0$ on $A$ and $+\infty$ on $A^c.$ For this function, $Q_ti_A=d_p^p(x,A)/t^{p-1}$ and one can choose $m(i_A)=0.$ Applying \eqref{eq-GRS11} gives the concentration property $\mathbf{CI}_p^\infty(\alpha)$.
\endproof

%%%%%%%%%%%%%%%%%%%%%%%%%%%%%%%%%%%%%%%%%%%%%%%%%%%%%%%%%%%%%%%%%%%%%%%%%%%%%%%%%%%%%%%%%%%%%%%%

\section{Poincar\'e inequality and concentration of measure}

This section is dedicated to the proof of our main result Theorem \ref{main result}, and to Corollary 
\ref{cor:obs diam}. Moreover, we shall explain how Theorem \ref{main result} can be used to improve
any non-trivial dimension concentration property to an exponential one.

\subsection{From dimension free concentration to the Poincar\'e inequality: proof of Theorem \ref{main result}}
\proof[Proof of Theorem \ref{main result}.]
Let $h:\X \to \R$ be a bounded Lipschitz function such that $\int h\,d\mu=0.$
For all $n\in \N^*$, define $f_n:\X^n\to \R^+$ by 
\[
f_n(x)=h(x_1)+\cdots+h(x_n),\qquad \forall x=(x_1,\ldots,x_n)\in \X^n.
\]
Our aim is to apply the Central Limit 
Theorem.
Applying \eqref{eq-GRS11} to $f_n$ with $t=1/\sqrt{n}$ and $r=\sqrt{n}u$, for some $u>0$, we easily arrive at
\begin{multline}\label{promo}
\mu^{n} \left( \frac{1}{\sqrt{n}\sigma_n}\sum_{i=1}^n \left[Q_{1/\sqrt{n}} h(x_i) -\mu (Q_{1/\sqrt{n}}h)\right] > \frac{1}{\sigma_n\sqrt{n}}m(f_n)+\frac{\sqrt{n}}{\sigma_n} \mu \left(h-Q_{1/\sqrt{n}}h \right) + \frac{u}{\sigma_n} \right)\\
\leq \alpha (\sqrt{u}),\end{multline}
where $\sigma_n^2 = \mathrm{Var}_{\mu} (Q_{1/\sqrt{n}} h)$ and $m(f_n)$ is a median of $f_n$ under $\mu^n$, that is to say any number $m\in \R$ such that $\mu^n(f\geq m)\geq1/2$ and $\mu^n(f_n\leq m)\geq 1/2.$

We deal with each term of \eqref{promo} separately. 
According to Point $(i)$ of Proposition \ref{prop:Qt}, we observe that $\sigma_n \to \sigma=\sqrt{\mathrm{Var}_\mu (h)}$, when $n$ goes to $\infty$
and according to Point $(ii)$ of Proposition \ref{prop:Qt}, that
\[
\limsup_{n\to +\infty}\sqrt{n}\mu \left(h-Q_{1/\sqrt{n}}h \right) \leq\frac{1}{4}\int |\nabla^- h|^2\,d\mu.
\]
On the other hand, let $m_n=m(f_n)/(\sqrt{n}\sigma)$. According to the Central Limit Theorem (Theorem \ref{TCL}) the law of the random variables $T_n=f_n/(\sqrt{n}\sigma)$ under $\mu^n$ converges weakly to the standard Gaussian. Since weak convergence implies the convergence of quantiles as soon as the limit distribution has a continuous repartition function (see for instance \cite[Lemma 21.2]{vdv98}), we have in particular, $m_n \to 0$ as $n\to\infty$.
%, we get, since $\overline \Phi $ is $(2\pi)^{-1/2}$-Lipschitz, 
%\begin{multline*}
%\frac{1}{2}-\delta_n\leq \mu^n\left(\frac{f_n}{\sqrt{n}\sigma} \geq m_n\right) - \delta_n \leq \overline{\Phi}(m_n-2\delta_n) 
% \leq 
%\overline{\Phi}(m_n+\delta_n)+3(2\pi)^{-1/2}\delta_n\\ \leq 
%\mu^n \left(\frac{f_n}{\sqrt{n}\sigma} >m_n\right) + (3(2\pi)^{-1/2}+1)\delta_n \leq 1/2 + (3(2\pi)^{-1/2}+1)\delta_n,
%\end{multline*}
%with $\delta_n=L(F_n,\Phi)\to 0$ as $n\to +\infty$, and $F_n(x)=\mu^n(f_n/(\sqrt{n}\sigma)\leq x)$.
%As a consequence $\overline{\Phi}(m_n-2\delta_n)\to1/2$, which implies that $m_n \to 0.$

Now, fix $\varepsilon>0$. According to the above observations,  \eqref{promo} yields, for any $u>0$ and any $n$ sufficiently large,
\begin{eqnarray}\label{vacances}
\qquad \quad \mu^{n} \left( \frac{1}{\sqrt{n}\sigma_n}\sum_{i=1}^n \left[Q_{1/\sqrt{n}} h(x_i) -\mu (Q_{1/\sqrt{n}}h)\right] > \frac{1+\varepsilon}{\sigma}\left( \varepsilon + \frac{1}{4}\int |\nabla^- h|^2\,d\mu  +u \right)  \right)
\leq \alpha (\sqrt{u}).
\end{eqnarray}
In order to apply (again) the Central Limit Theorem (Theorem \ref{TCL}), introduce the following random variables 
\[
\widetilde T_n(x)=\sum_{i=1}^n \frac{1}{\sqrt{n}\sigma_n} \left[Q_{1/\sqrt{n}} h(x_i) -\mu (Q_{1/\sqrt{n}}h)\right]
\]
 under $\mu^n$. Since $h$ is bounded $Q_{1/\sqrt{n}} h$ is uniformly bounded in $n$, and by Lebesgue theorem, as $n$ goes to $\infty$, we see that the Lindeberg condition \eqref{lindeberg} is verified:\[
\int \frac1{\sigma_n^2}(Q_{1/\sqrt{n}} h-\mu(Q_{1/\sqrt{n}} h))^2\1_{|Q_{1/\sqrt{n}} h-\mu(Q_{1/\sqrt{n}} h)|>t\sqrt n \sigma_n} d\mu \to0 \qquad \hbox{ as } n \to +\infty.
\]
Therefore, letting $n$ go to $+\infty$, and $\varepsilon$ to $0$, by continuity of $\overline \Phi$, the inequality \eqref{vacances} provides  
\[
\overline{\Phi} \left(\frac{\int  |\nabla^- h|^2\,d\mu}{4\sigma} + \frac{u}{\sigma}  \right ) \leq \alpha (\sqrt{u}), \qquad \forall u \geq 0.
\]
Let $u\geq 0$ be such that $\alpha (\sqrt{u}) < 1/2$ and $k(u) := \overline{\Phi}^{-1}(\alpha (\sqrt{u}))>0$. 
We easily get from the latter inequality that 
\[
k(u)\sqrt{\mathrm{Var}_\mu (h)} \leq u + \frac{1}{4}\int |\nabla^- h|^2\,d\mu. 
\]
Replacing $h$ by $s h$, $s >0$, and taking the infimum, we arrive at 
\[
k(u) \sqrt{\mathrm{Var}_\mu (h)} \leq \inf_{s>0}\left\{\frac{u}{s} +\frac{s}{4}  \int |\nabla^- h|^2\,d\mu \right\}=\sqrt{u \int  |\nabla^- h|^2\,d\mu}.
\]
Optimizing over $u$, one concludes that Poincar\'e inequality \eqref{Poincare} is satisfied with the constant $\lambda$ announced in Theorem \ref{main result}.

Now let $h:\X\to \R$ be an unbounded Lipschitz function. Consider the sequence of functions $h_n=(h\vee -n)\wedge n$, $n\in\N^*$ converging pointwise to $h$. For all $n\in\N^*$, $h_n$ is bounded and Lipschitz, and it is not difficult to check that 
\begin{equation}\label{eq:truncation}
|\nabla^-h_n|(x)= 0\ \text{ if }x\in\{h\leq -n\}\cup \{h>n\}\quad \text{and}\quad |\nabla^-h_n|(x)= |\nabla^-h_n|(x)\ \text{ if }x\in\{-n<h\leq n\}.
\end{equation}
In particular, the sequence $|\nabla^-h_n|$ converges monotonically to $|\nabla^-h|$. Applying Fatou's lemma and the monotone convergence theorem, we obtain
\[
\lambda\iint (h(x)-h(y))^2\,\mu(dx)\mu(dy)\leq \lambda\liminf_{n\to\infty}\mathrm{Var}_\mu(h_n)\leq \liminf_{n\to \infty}\int  |\nabla^-h_n|^2\,d\mu = \int  |\nabla^-h|^2\,d\mu, 
\]
which completes the proof of Theorem \ref{main result}.
\endproof

\subsection{Poincar\'e inequality and boundedness of observable diameters of product probability spaces}
In this section we prove Corollary \ref{cor:obs diam}.

\proof[Proof of Corollary \ref{cor:obs diam}]
Assume first that $\mu$ satisfies the Poincar\'e inequality \eqref{Poincare} with the optimal constant $\lambda$. Then according to Theorem \ref{Gromov Milman}, $\mu$ satisfies $\mathbf{CI}_2^\infty(\alpha)$ with the concentration profile $\alpha(r)=be^{-\sqrt{\lambda}r}$, where $a,b$ are universal constants ($b\geq 1/2$). According to the first part of Lemma \ref{lem:FS12} (applied to the metric probability space $(\X^n,d_2,\mu^n)$), it follows that for all $n\in \N^*$, $\mathrm{Obs\, Diam}(\X^n,d_2,\mu^n, t) \leq 2\frac{\log(4b/t)}{a\sqrt{\lambda}},$ for all $t\leq 1$
and thus 
\[
r_\infty(t)\sqrt{\lambda}\leq a'\log(b'/t),\qquad \forall t\leq 1
\] 
for some universal constants $a',b'.$ 

Conversely, assume that $0<r_\infty(t_o)<\infty$ for some $t_o\in (0,1/2).$ According to the second part of Lemma \ref{lem:FS12}, $\mu$ satisfies $\mathbf{CI}_2^\infty(\beta_{t_o,r_\infty(t_o)})$, where the minimal profiles $\beta$ are defined in \eqref{minimal profile}. According to Theorem \ref{main result}, it
follows that $\mu$ satisfies the Poincar\'e inequality with an optimal constant $\lambda>0$ such that
\[
\sqrt{\lambda}r_\infty(t_o) \geq \overline{\Phi}^{-1}(t_o).
\]
According to the first step, we conclude that $r_\infty(t)<\infty$ for all $t\leq 1$, and so the inequality above is true for all $t\in (0,1/2).$
\endproof

\subsection{Self improvement of dimension free concentration inequalities}\label{self improvement}
The next result shows that a non-trivial dimension free concentration inequality can always be upgraded into an inequality with an exponential decay. This observation goes back to Talagrand \cite[Proposition 5.1]{Tal91}. 
\begin{cor}\label{Cor Tal}
If $\mu$ satisfies $\mathbf{CI}_2^\infty(\alpha)$ with a profile $\alpha$ such that $\alpha(r_o)<1/2$ for some $r_o$, then it satisfies $\mathbf{CI}_2^\infty$ with an exponential concentration. More explicitly, it satisfies the dimension free concentration property with the profile $\tilde{\alpha}(r)=be^{-a\sqrt{\lambda}r}$, where $a,b$ are universal constants and 
\[
\sqrt{\lambda}= \sup \left\{\frac{\overline{\Phi}^{-1} \left(\alpha(r)\right)}{r}; r>0 \text{ s.t } \alpha (r) < 1/2  \right\}.
\]
\end{cor}
This result is an immediate corollary of Theorem \ref{main result} and Theorem \ref{Gromov Milman}.

In \cite{Tal91} this result was stated and proved only for probability measures on $\R$. We thank E. Milman for mentioning to us that the argument was in fact more general. 
For the sake of completeness, we extend below Talagrand's argument in a very general abstract framework. For a future use, we only assume that the dimension free concentration property holds on a good subclass of sets. We refer to the proof of Proposition \ref{Gromov Milman convex} where this refinement will be used (the subclass of sets being the class of convex sets). 

\begin{prop}\label{Cor Tal 2}
Let $(\X,d)$ be a complete separable metric space, $p\geq 1$ and for all $n\in \N^*$ let $\mathcal{A}_n$ be a class of Borel sets in $\X^n$ satisfying the following conditions:
\begin{itemize}
\item[(i)] For all $n\in \N^*$ and $r\geq 0$, if $A\in \mathcal{A}_n$ then $A_{r,p}\in \mathcal{A}_n.$
\item[(ii)] If $A\in \mathcal{A}_m$, then $A^n\in \mathcal{A}_{nm}.$ 
\end{itemize}
Suppose that a Borel probability measure $\mu$ on $\X$ satisfies the following dimension free concentration property: there exists $r_o>0,a_o\in[0,1/2)$ such that for all $n\in \N^*$, 
\[
\mu^n(A_{r_o,p})\geq 1-a_o,\qquad \forall A \in \mathcal{A}_n \text{ s.t. } \mu^n(A)\geq 1/2.
\]
Then, for any $\gamma \in ( -\log(1-a_o)/\log(2), 1)$, there exists $c\in [1/2,1)$ depending only on $\gamma$ and $a_o$ such that for all $n\in \N^*$,
\[
\mu^n(A_{r,p}) \geq 1- \frac{1-c}{\gamma} \gamma^{r/r_o},\qquad \forall r\geq0,\qquad \forall A \in \mathcal{A}_n \text{ s.t. } \mu^n(A)\geq c.
\]
\end{prop}
Note in particular that in the case $p=2$ (and $\mathcal{A}_n$ the class of all Borel sets), we recover the conclusion of Corollary \ref{Cor Tal} with slightly less accurate constants. 

\proof
Given $A\in \mathcal{A}_1$, it holds $\left(A^n\right)_{r_o,p}\subset \left(A_{r_o}\right)^n$ and, according to (i) and (ii), both sets belong to $\mathcal{A}_n.$ Therefore, if $\mu(A)\geq (1/2)^{1/n}$, it holds $\mu(A_{r_o}) \geq (1-a_o )^{1/n}.$
Now, let $A\in \mathcal{A}_1$ be such that $\mu(A)\geq1/2$ and let $n_A$ be the greatest integer $n\in \N^*$ such that
$\mu(A)\geq (1/2)^{1/n}.$ By definition of $n_A$, 
\[
\frac{\log(2)}{\log(1/\mu(A))}-1< n_A\leq \frac{\log(2)}{\log(1/\mu(A))}.
\]
According to what precedes, 
\[
\mu(A_{r_o}^c)\leq 1-(1-a_o )^{1/n_A}\leq 1-\exp\left(\frac{\log(1-a_o )\log(1/\mu(A))}{\log(2)-\log(1/\mu(A))}\right).
\]
The function $\varphi(u)=\exp\left(\frac{\log(1-a_o )\log(1/u)}{\log(2)-\log(1/u)}\right)$ satisfies
\[
\varphi(u)=1-\frac{\log(1-a_o )}{\log(2)}(u-1)+o(u-1),
\]
when $u\to 1.$ So $\frac{1-\varphi(\mu(A))}{1-\mu(A)} \to -\frac{\log(1-a_o )}{\log(2)}\in (0,1),$ when $\mu(A)\to1.$ Therefore, if $\gamma$ is any number in the interval $(-\frac{\log(1-a_o )}{\log(2)},1)$, there exists $c>1/2$ (depending only on $\gamma$) such that for all $A \in \mathcal{A}_1$ with $\mu(A)\geq c$ it holds
\[
\mu(A_{r_o}^c)\leq \gamma \mu(A^c).\]
Iterating (which is possible thanks to $(i)$ and the easy to check property $\left(A_{r_1,p}\right)_{r_2,p}\subset A_{r_1+r_2,p}$) yields
\[
\mu(A_{kr_o}^c)\leq \gamma^k \mu(A^c),\qquad \forall k\in \N^*.
\]
It follows easily that for all $u\geq0$, 
\[
\mu(A_u^c)\leq ((1-c)/\gamma) \gamma^{u/r_o}.
\]
Applying the argument above to the product measure $\mu^p$, $p\in \N^*$, gives the conclusion.
\endproof

%%%%%%%%%%%%%%%%%%%%%%%%%%%%%%%%%%%%%%%%%%%%%%%%%%%%%%%%%%%%%%%%%%%%%%%%%%%%%%%%%%%%%%%%%%%%%%%%%%%%

\section{ From Poincar\'e inequality to exponential concentration: proof of Theorem \ref{Gromov Milman}} \label{sec:gromov-milman}
In this section, we give a proof of Theorem \ref{Gromov Milman}. Its conclusion is very classical in, say, a Euclidean setting. But to deal with the general metric space framework requires some additional technical ingredients that we present now.

\subsection{Technical preparation}\label{sec:tec}

In order to regularize Lipschitz functions, we shall introduce an approximate sup-convolution operator. More precisely, for all $\varepsilon>0$, and for all function $f:\X^n\to \R\cup\{-\infty\}$, we define the (approximate) sup-convolution operator by
 \begin{equation} \label{defRt}
R_{\varepsilon} f(x):=\sup_{y\in \X^n}\left\{ f(y) -\sqrt{\varepsilon +d_2^2(x,y)} \right\},\qquad x\in \X^n .
 \end{equation}

The next lemma collects some useful properties about $R_\varepsilon$. 
 \begin{lem}\label{approx} Let $f:\X^n\to \R\cup\{-\infty\}$ be a function taking at least one finite value and $\varepsilon >0$. 
 Then 
 \begin{enumerate}[label={\bfseries (\roman*)}]
\item If $R_\varepsilon f(x_o)<\infty$ for some $x_o\in \X^n$, then $R_\varepsilon f$ is finite everywhere and is $1$-Lipschitz with respect to $d_2$. Moreover, it holds  
$$\displaystyle \sum_{i=1}^n|\nabla_i^- R_\varepsilon f(x)|^2\leq 1,\qquad \forall x\in \X^n.$$
\item If $f$ is $1$-Lipschitz, then for all $x\in \X^n$, 
$$f(x)-\sqrt \varepsilon\leq R_\varepsilon f(x)\leq f(x).$$
\item If $(\X,d)$ is a Banach space and $f$ is convex, then $R_\varepsilon f$ is also convex.
\end{enumerate}

\end{lem}

\begin{proof}We fix $\varepsilon>0.$

Point $(iii)$ follows easily from the fact that $R_\varepsilon f$ is a supremum of convex functions, since by a simple change of variable
\[
R_\varepsilon f(x)=\sup_{z\in \X^n}\left\{ f(x-z) -\sqrt{\varepsilon +\|z\|_2^2} \right\},\qquad x\in \X^n.
\]

Point $(ii)$ is also easy. Indeed, the first inequality follows  by choosing $y=x$. For the second inequality, observe that, since $f$ is 1-Lipschitz,
\[
R_\varepsilon f(x)-f(x)=\sup_{y\in \X^n}\left\{ f(y)-f(x) -\sqrt{\varepsilon +d_2^2(x,y)} \right\}\leq \sup_{r\geq 0}\left\{r -\sqrt{\varepsilon +r^2}\right\}=0.
\]

Now we turn to the proof of Point $(i)$. 

The first part of the statement follows from the fact that $R_\varepsilon f$ is a supremum of $1$-Lipschitz functions. To prove the other part, we need to fix some notation.
For $x=(x_1,\dots,x_n) , z=(z_1,\dots,z_n) \in \X^n$ and $i \in \{1,2,\ldots,n\}$, 
we set $$\bar{x}^iz=(x_1,\ldots,x_{i-1},z_i,x_{i+1},\ldots,x_n).$$
Also, we set $\theta(u):=\sqrt{\varepsilon+u}$, $u \in \R$, so that
$R_\varepsilon f(x)=\sup_y \{f(y)-\theta(d_2^2(x,y))\}$ (observe that $\theta$ is concave). 

Fix $x=(x_1,\dots,x_n) \in \X^n$ and a parameter  $\eta \in (0,1)$ that will be chosen later on and consider $z=(z_1,\ldots, z_n)\in \X^n$ such that $z_i\neq x_i$ for all $i\in\{1,\ldots,n\}$. We assume that $R_\varepsilon f$ is everywhere finite.
Hence, since $\eta\min_{1\leq i\leq n} d(x_i,z_i)>0$, 
there exists $\hat y=\hat y(x,z,\eta)$ such that 
\[
R_\varepsilon f(x) \leq  f(\hat y) -\sqrt{\varepsilon +d_2^2(x,\hat y)} 
+ \eta \min_{1\leq i\leq n} d(x_i,z_i). 
\]
As a consequence, using that $\theta$ is concave, for all $1\leq i\leq n$, we have
\begin{align*}
[R_\varepsilon f(\bar x^i z_i)-R_\varepsilon f(x)]_-
&=
[R_\varepsilon f(x)-R_\varepsilon f(\bar x^i z_i)]_+ 
\leq 
\left[\theta(d_2^2(x^i z_i,\hat y))-\theta(d_2^2(x,\hat y))\right]_+ 
+\eta\min_{1\leq i\leq n} d(x_i,z_i)\\
& \leq 
\left[d_2^2(x^i z_i,\hat y) - d_2^2(x,\hat y) \right]_+ \theta'(d_2^2(x,\hat y))
+\eta\min_{1\leq i\leq n} d(x_i,z_i)\\
& =
\left[d(\hat y_i,z_i)-d(\hat y_i,x_i)\right]_+\left(d(\hat y_i,z_i)+d(\hat y_i,x_i)\right)
\theta'(d_2^2(x,\hat y))
+\eta\min_{1\leq i\leq n} d(x_i,z_i) \\
&\leq
2 d(z_i,x_i) d(\hat y_i,z_i)
\theta'(d_2^2(x,\hat y))
+\eta\min_{1\leq i\leq n} d(x_i,z_i)
\end{align*}
where, in the last line we used that the positive part $\left[d(\hat y_i,z_i)-d(\hat y_i,x_i)\right]_+$ guarantees that $d(\hat y_i,x_i) \leq d(\hat y_i,z_i)$ and the triangular inequality.

Using the Cauchy-Schwarz inequality, it follows for any $\delta>0$ that
\begin{align*}
\frac {[R_\varepsilon f(\bar x^i z_i)-R_\varepsilon f(x)]_-^2}{d^2(x_i,z_i)}
& \leq 
\left( 2 d(\hat y_i,z_i)
\theta'(d_2^2(x,\hat y))
+\eta \right) ^2  
 \leq
(1+\delta) 4 d^2(\hat y_i,z_i) \theta'(d_2^2(x,\hat y))^2 + \left(1+\frac{1}{\delta}\right)\eta^2 .
\end{align*}
Therefore, using the triangular and the Cauchy-Schwarz inequalities, we get for any $\delta>0$,
\begin{align*}
\sum_{i=1}^n \frac {[R_\varepsilon f(\bar x^i z_i)-R_\varepsilon f(x)]_-^2}{d^2(x_i,z_i)}
& \leq 
(1+\delta) 4 d_2^2(\hat y,z) \theta'(d_2^2(x,\hat y))^2 + n\left(1+\frac{1}{\delta}\right)\eta^2 \\
& \leq 
(1+\delta) 4 (d_2(\hat y,x) + d_2(x,z))^2 \theta'(d_2^2(x,\hat y))^2 + n\left(1+\frac{1}{\delta}\right)\eta^2 \\
& \leq 
(1+\delta)^2 4 d_2^2(\hat y,x)  \theta'(d_2^2(x,\hat y))^2 \\
& \quad
+ (1+\delta)\left(1+\frac{1}{\delta}\right) 4 d_2^2(x,z)^2 \theta'(d_2^2(x,\hat y))^2 
+ n\left(1+\frac{1}{\delta}\right)\eta^2 \\
& \leq (1+\delta)^2 + (1+\delta)\left(1+\frac{1}{\delta}\right) \frac {d_2^2(x,z)^2}{\varepsilon}
+ n\left(1+\frac{1}{\delta}\right)\eta^2,
\end{align*}
using that $4u^2\theta'(u^2)^2 \leq 1$ and $\theta'(u^2)^2\leq 1/(4\varepsilon).$
The expected result follows at once taking the limits
$z_i \to x_i$, and then $\eta \to 0$ and $\delta \to 0$.
This ends the proof of the lemma.
\end{proof}

The next ingredient in the proof of Theorem \ref{Gromov Milman} is the so-called Herbst argument
that we explain now. We recall the following notation: for any $f \colon \X^n \to \R$, set
\[
m(f)=\inf\left\{ m\in \R; \mu^n(f\leq m)\geq 1/2 \right\}.
\]
\begin{prop}\label{Herbst}
Assume that $\mu$ satisfies the Poincar\'e inequality \eqref{Poincare} with a constant $\lambda>0$. Then, there exist universal constants $a,b>0$ such that  then for all $n\in \N^*$, it holds
\[
\mu^n\left(f>m\left(  f\right) + \sqrt{\frac { 2} {\lambda }} +r\right)\leq b\exp\left(-a\sqrt{\lambda}r\right),\qquad \forall r\geq0
\]
for all function $f:\X^n\to \R$ such that 
\begin{equation}\label{eq:gradient-}
\sum_{i=1}^n |\nabla^-_if|^2(x)\leq 1,\qquad \forall x\in \X^n.
\end{equation}
The same conclusion holds if the function $f$ satisfies
\begin{equation}\label{eq:gradient+}
\sum_{i=1}^n |\nabla^+_if|^2(x)\leq 1,\qquad \forall x\in \X^n.
\end{equation}
\end{prop}

\proof
It is well known that the Poincar\'e inequality tensorizes properly (see \textit{e.g.}\ \cite[proposition 5.6]{Ledoux-book}). Indeed, recall that 
  for all $n\in \N^*$ and for all product probability measure $\mu^n$
\begin{equation*}%
\lambda\mathrm{Var}_{\mu^n}(g)\leq \int\sum_{i=1}^n \mathrm{Var}_{\mu}(g_i)\,\mu^n(dx),
\end{equation*}
where  $g_i(x_i):=g(x_1,\ldots,x_{i-1},x_i,x_{i+1},\ldots,x_n)$, with $x_1,\ldots,x_{i-1},x_{i+1},\ldots,x_n$ fixed.
Therefore,  if $\mu$ satisfies \eqref{Poincare}, then the product probability measure $\mu^n$ satisfies
\begin{equation}\label{Poincare tens}
\lambda\mathrm{Var}_{\mu^n}(g)\leq \int\sum_{i=1}^n |\nabla_i^-g|^2(x)\,\mu^n(dx),
\end{equation}
for all function $g:\X^n \to \R$ that is Lipschitz in each coordinate.

Let $f:\X^n\to\R$ be bounded and such that \eqref{eq:gradient-} holds, and define $Z(s)=\log \int e^{s f}\,d\mu^n$, for all $s\geq0$. Applying \eqref{Poincare tens} to $g=e^{s f}$ (which is still bounded and Lipschitz) and using \eqref{eq:gradient-} yields easily to
\[
\lambda\left[\int e^{2s f}d\mu^n - \left(\int e^{s f}\,d\mu^n\right)^2\right] \leq s^2\int e^{2sf}\,d\mu^n .
\]
Thus
\[
\log(1-s^2/\lambda) + Z(2s) \leq 2Z(s),\qquad \forall 0\leq s \leq \sqrt{\lambda}.
\]
According to H\"older's inequality, the function $Z$ is convex. Therefore,
\[
Z(2s)\geq Z(s) + Z'(s)s.
\]
As a result,
\[
\log(1-s^2/\lambda) + Z'(s)s \leq Z(s),\qquad \forall 0\leq s \leq \sqrt{\lambda},
\]
and so
\[
\frac{d}{ds}\left( \frac{Z(s)}{s}\right) \leq -\frac{\log(1-s^2/\lambda)}{s^2},\qquad \forall 0< s \leq \sqrt{\lambda}.
\]
Since $Z(s)/s\to \int f\,d\mu^n$ when $s\to0$, we conclude that
\[
\int e^{s(f-\int f\,d\mu^n)}d\mu^n \leq \exp\left(\frac{s}{\sqrt{\lambda}} \int_0^{s/\sqrt{\lambda}} \frac{-\log(1-v^2)}{v^2}\,dv \right),\qquad \forall s\leq \sqrt{\lambda}.
\]
Taking $s=\sqrt{\lambda}/2$, we easily get, by the (exponential) Chebychev Inequality,
\begin{eqnarray}\label{hourra}
\mu^n\left(f-\int f\,d\mu^n >r\right)\leq be^{-\frac{\sqrt{\lambda}}{2}r},\qquad \forall r\geq0,
\end{eqnarray}
with $b=\exp\left( \frac{1}{2}\int_0^{1/2}\frac{-\log(1-v^2)}{v^2}\,dv\right).$

Now, since the probability measure  $\mu^n$  satisfies the Poincar\'e inequality with constant $\lambda>0$,    
\eqref{eq:gradient-} implies that 
\[\lambda\mathrm{Var}_{\mu^n}( f)\leq 1.\]
Therefore, by Markov's inequality, it holds, for all $\varepsilon>0$
 \begin{eqnarray*}
\mu^n\left( f\leq \int  f\,d\mu^n  -\sqrt{\frac { 2+\varepsilon} {\lambda}} \right) 
\leq \frac{\lambda}{2+\varepsilon} \,\mathrm{Var}_{\mu^n}( f) < \frac12.
\end{eqnarray*}
Hence, according to this definition of $m(f)$, it follows easily  that 
\[
m\left(  f\right)   \geq \int  f\,d\mu^n  -\sqrt{\frac { 2} {\lambda} }  .
\]
This inequality together with \eqref{hourra} provides the expected deviation inequality 
\begin{equation}\label{youhou}
\mu^n\left(f>m\left(  f\right) + \sqrt{\frac { 2} {\lambda} } +r\right)\leq be^{-\frac{\sqrt{\lambda}}{2}r},\qquad \forall r\geq0.
\end{equation}
Now suppose that $f:\X^n\to \R$ is Lipschitz but not bounded. Consider the function $f_n=(f\vee -n) \wedge n$, $n\in \N^*.$ 
Applying \eqref{eq:truncation} componentwise, we see that $|\nabla_i^-f_n|\leq |\nabla_i^-f|$ for all $i\in\{1,2,\ldots,n\}$. Therefore, $f_n$ satisfies \eqref{eq:gradient-}. Applying \eqref{youhou} to $f_n$, and letting $n$ go to $\infty$, one sees easily
that \eqref{youhou} still holds for the unbounded Lipschitz function $f$. 

Exactly the same proof works under the condition \eqref{eq:gradient+} (using the equivalent inequality \eqref{Poincare+})
\endproof

\subsection{Proof of Theorem \ref{Gromov Milman}}
Thanks to the previous section (Section \ref{sec:tec}), we are now in position to prove Theorem \ref{Gromov Milman}.
\proof[Proof of Theorem \ref{Gromov Milman}]
Let $A$ be a measurable subset of $\X^n$ of measure $\mu^n(A)\geq 1/2$ and define for all $\varepsilon>0$,
$f_{A,\varepsilon}(x):=\sqrt{\varepsilon+d_2^2(x,A)}$, $x \in \X^n$. By definition of $R_\varepsilon$ given in \eqref{defRt}, it holds $f_{A,\varepsilon}=-R_\varepsilon i_A$, where $i_A:\X^n\to \{-\infty; 0\}$ is the function defined by $i_A(x)=0$, if $x\in A$ and $i_A(x)=-\infty$ otherwise. 
According to Point $(ii)$ of Lemma \ref{approx}, we see that $f_{A,\varepsilon}$ satisfies Condition \eqref{eq:gradient+} of Proposition \ref{Herbst}. Hence, observing that $m(f_{A,\varepsilon})=\sqrt{\varepsilon}$, it follows from Proposition \ref{Herbst} that
\[
\mu^n\left(\sqrt{\varepsilon+d_2^2(x,A)}>\sqrt{\varepsilon} + \sqrt{\frac { 2} {\lambda} } +r\right)\leq b\exp\left(-a\sqrt{\lambda}r\right),\qquad \forall r\geq0 .
\]
Letting $\varepsilon$ go to $0$ yields (after a change of variable) to,
 \begin{eqnarray*}
1-\mu^n(A_{r,2}) \leq b' \exp(-a\sqrt{\lambda} r),\quad \forall r\geq 0,
\end{eqnarray*}
with $b'=\max(b,1/2)e^{\sqrt 2 a}$. The proof of Theorem \ref{Gromov Milman} is complete.
\endproof

\section{Poincar\'e inequality for convex functions and dimension free convex concentration property.}

In this final section, we investigate the links between the dimension free concentration property  $\mathbf{CI}^\infty_2(\alpha)$ restricted to convex sets, and the Poincar\'e inequality restricted to the class of convex functions (see below for a precise definition)

\subsection{Convexity and convex concentration on geodesic spaces}
To deal with convexity properties, we shall assume throughout the section that the metric space $(\X,d)$ is  geodesic. This means that any two points $x$ and $y $ of $\X$ can be connected  by  at least one  constant-speed continuous curve in $\X$: \textit{i.e.}\ for all $x,y \in \X$, there exists $(x_t)_{t\in[0,1]}$ in $\X$ satisfying $x_0=x$, $x_1=y$, and for all $s,t\in[0,1]$, $d(x_s,x_t)=|t-s| d(x_0,x_1)$. Such a curve is called a 
(constant-speed) geodesic joining $x$ to $y$.

\begin{defi} Let $(X,d)$ be a geodesic space.
\begin{enumerate}[label={\bfseries (\roman*)}]
\item A  set $A\subset \X$ is said to be convex if for all  $x_0, x_1 \in A$ and all geodesics $(x_t)_{t\in[0,1]}$ joining $x_0$ to $x_1$, it holds $x_t\in A$, $\forall t \in [0,1]$.
\item A function $f:\X\to \R\cup\{+\infty\}$ is said to be convex if for all geodesics $(x_t)_{t\in[0,1]}$ in $\X$,
\[f(x_t)\leq (1-t)f(x_0) +t f(x_1).\]
\end{enumerate}
\end{defi}

Accordingly, one will say that $\mu \in \mathcal{P}(\X)$ satisfies the dimension free convex concentration property with the concentration profile $\alpha$ and with respect to the $\ell_2$ product structure (in short $\mathbf{CCI}_2^\infty(\alpha)$), if for all convex subset $A\subset \X^n$, with $\mu^n(A)\geq 1/2,$ 
\begin{equation}\label{eq:CCI}
\mu^n(A_{r,2})\geq 1-\alpha(r),\qquad \forall r\geq 0.
\end{equation}

As for $\mathbf{CI}_2^\infty(\alpha)$  in Proposition \ref{prop-GRS11}, the convex concentration property can be characterized using the inf-convolution operator $Q_t$.
%(we omit the proof that is a straight forward adaptation of the one of Proposition \ref{prop-GRS11}). 
\begin{prop}\label{concconvfunc}
Let $\mu \in \mathcal{P}(\X)$; $\mu$ satisfies $\mathbf{CCI}_2^\infty(\alpha)$ if and only if for all $n\in \N^*$ and for all convex  function $f:\X^n \to \R\cup\{+\infty\}$ bounded from below and such that $\mu^n(f=+\infty)<1/2$, it holds
\begin{equation}\label{eq-GRS11convex}
 \mu^n(Q_t f > m(f) +r)\leq \alpha(\sqrt{tr}),\quad \forall r,t>0,
 \end{equation}
where $m(f)=\inf\{m\in\R;\mu^n(f\leq m(f)) \geq 1/2\}.$
\end{prop}
\proof 
By the following two observations the proof of the above proposition (that we omit) becomes essentially identical to the one of  Proposition \ref{prop-GRS11} in section \ref{link}:
$(a)$ By definition, if $f:\X\to \R\cup\{+\infty\}$ is convex then any level set $\{f\leq r\}$, $r\in\R$, is convex; $(b)$ for any convex set $A$, the function $i_A$, that equals 0 on $A$ and $+\infty$ on $\X\setminus A$, is convex.
\endproof
For technical reasons, we will have to distinguish between the following two versions of the Poincar\'e inequality in restriction to convex functions (in short \emph{convex Poincar\'e}):
\begin{equation}\label{Poincare convex -}
\lambda\mathrm{Var}_\mu(f)\leq \int |\nabla^-f|^2\,d\mu,\qquad \forall f \text{ Lipschitz and convex}
\end{equation}
and
\begin{equation}\label{Poincare convex +}
\lambda\mathrm{Var}_\mu(f)\leq \int |\nabla^+f|^2\,d\mu,\qquad \forall f \text{ Lipschitz and convex}.
\end{equation}

The argument used in Point (3) of Remark \ref{rem Poincare} to prove the equivalence between \eqref{Poincare} and \eqref{Poincare+} in the usual setting does not work anymore since if $f$ is convex, $-f$ is no more convex. 
However, if $(\X,d)$ is finite dimensional Banach space or a smooth Riemannian manifold and if one assumes that $\mu$ is absolutely continuous with respect to the Lebesgue (or the volume) measure, by Rademacher's theorem, both gradients in the definition of \eqref{Poincare convex -} and \eqref{Poincare convex +} are equal, except on a set of $\mu$-measure 0, which, in turn, guarantees that  
\eqref{Poincare convex -} and \eqref{Poincare convex +} are equivalent.

\subsection{Dimension free convex concentration implies the convex Poincar\'e inequality \ref{Poincare convex -}}Starting from the functional characterization of $\mathbf{CCI}_2^\infty(\alpha)$ stated in Proposition \ref{concconvfunc} and following the lines of the proof of Theorem \ref{main result}, we shall obtain the first main theorem of this section (a counterpart of Theorem \ref{main result} in the convex situation). 

\begin{thm}\label{main result convex} If $\mu$ satisfies the dimension free convex concentration property $\mathbf{CCI}_2^\infty(\alpha)$, then  $\mu$ satisfies the convex Poincar\'e inequality 
\[
\lambda\mathrm{Var}_\mu(f)\leq \int |\nabla^-f|^2\,d\mu,
\]
for all locally Lipschitz\footnote{By \emph{locally Lipschitz}, we mean Lipschitz on every ball of finite radius.}and convex function with finite variance with the constant $\lambda$ defined by
\[
\sqrt{\lambda}= \sup\left\{\frac{ \overline{\Phi}^{-1} \left(\alpha(r)\right)}{r}; r> 0 \text{ s.t } \alpha(r)\leq 1/2  \right\}.
\]
If moreover $\int_0^{\infty} r\alpha(r)\,dr<\infty$ then the variance of all convex and  Lipschitz functions is finite and thus $\mu$ verifies \eqref{Poincare convex -}.  
\end{thm}

\proof The proof is very similar to the proof of Theorem  \ref{main result}, but one needs to take care of the technical difficulties coming from the restriction to convex sets/functions. We give here only the main lines and omit most of the details.
Also, we may use the notation of the proof of Theorem  \ref{main result}.

Let $h:\X \to \R$ be a locally Lipschitz convex function such that $\int h\,d\mu=0$ and $\int h^2\,d\mu<+\infty$.
For all $n\in \N^*$, define $f_n:\X^n\to \R^+$ by 
\[
f_n(x)=h(x_1)+\cdots+h(x_n),\qquad \forall x=(x_1,\ldots,x_n)\in \X^n.
\]
We first observe that $f_n$ is convex. Indeed, this point is an easy  consequence of the fact that $(x_t)_{t\in [0,1]}$ is a geodesic in $(\X^n,d_2)$  if and only if for all $i\in \{1\ldots,n\}$, $(x_{i,t})_{t\in [0,1]}$ are  geodesics in $(\X,d)$ (where $x_t=(x_{1,t},\ldots,x_{n,t})$). Therefore we may apply Proposition \ref{concconvfunc} to the function $f_n$. 

For the next step of the proof,  we need some analogue of Proposition \ref{prop:Qt} for convex locally Lipschitz functions $h \colon \X\to \R$ (not necessarily globally Lipschitz). From  the  
 definition of the convexity property of $h$, one may easily check that  for all  $x,y\in \X$ 
\[
h(x)-h(y)\leq d(x,y)|\nabla^- h(x)|.\]
It follows that for all $t>0$,
\begin{align*}
0\leq\frac{h(x)-Q_th(x)}{t}
&=
\sup_{y\in\X}\left\{\frac{h(x)-h(y)}{t} - \frac{d^2(x,y)}{t^2}\right\} 
 \leq 
\sup_{y\in\X}\left\{|\nabla^- h(x)| \frac{d(x,y)}{t} - \frac{d^2(x,y)}{t^2}\right\} \\
& =
\sup_{r \in \mathbb{R}}\left\{|\nabla^- h(x)| r - r^2\right\}
\leq \frac{1}{4}|\nabla^- h(x)| ^2<\infty.
\end{align*}
This implies in particular that $Q_th(x)\to h(x)$ for all $x\in \X$. Moreover, it holds
 \[
\sqrt{n}\mu \left(h-Q_{1/\sqrt{n}}h \right) \leq\frac{1}{4}\int |\nabla^- h|^2\,d\mu,\qquad \forall n\in\N^*.
\]
Let us first assume that $h$ is bounded from below. Then it holds, $|Q_{1/\sqrt n}h|\leq |h|+|\inf h|$ for all $n\in\N^*$. Since $\int h^2\, d\mu<\infty$, applying the dominated convergence theorem yields to
 \[
 \lim_{n\to+\infty} \sigma_n^2=\lim_{n\to+\infty} \mathrm{Var}_{\mu} (Q_{1/\sqrt{n}} h)=\mathrm{Var}_{\mu} ( h)=\sigma^2.
 \]
 Since $\int h^2\,d\mu<+\infty$, we may  show that  any  median of  $m_n$ of $f_n/(\sqrt n \sigma)$ tends to 0 as $n$ goes to $+\infty$. As a consequence, \eqref{vacances} holds for $n$ sufficiently large. The rest of the proof is identical to the one of Theorem  \ref{main result}:
we  apply the Central Limit Theorem (Theorem \ref{TCL}), observing that the Lindeberg condition holds since $|Q_{1/\sqrt n}h|\leq |\inf h|+|h|$ for all $n\in\N^*$, and $\int h^2\, d\mu<\infty$.
This proves the claim in the case where $h$ is assumed bounded from below. To show that the Poincar\'e inequality still holds if $h$ is not bounded from below, one considers the approximation sequence $h_n=h\vee -n$, $n\in\N^*$ (note that $h_n$ is convex) and one follows the same line of reasoning as in the end of the proof of Theorem \ref{main result}. 

Finally, to complete the proof, observe that if $f:\X\to \R$ is convex and $1$-Lipschitz, then $A=\{f\leq m(f)\}$ is a convex set with $\mu(A)\geq 1/2$, and since $A_{r,2}\subset \{f\leq m+r\}$, we conclude from $\mathbf{CCI}_2(\alpha)$ that
\[
\mu(f>m(f)+r)\leq \alpha(r),\qquad \forall r\geq 0.
\]
Since $\int_0^{\infty} r\alpha(r)\,dr<\infty$, an integration by part shows that $\int [f-m(f)]_+^2\,d\mu<\infty.$ Therefore, if $f$ is convex, $1$-Lipschitz and bounded from below, then $f$ has a finite variance and so applying the first part of the proof, we conclude that $\mathrm{Var}_\mu(f)\leq 1/\lambda.$ Using the same truncation as above, we see that this inequality extends to all convex and $1$-Lipschitz functions. This complete the proof.
\endproof

\subsection{Convex Poincar\'e inequality implies exponential dimension free convex concentration}To get the converse, namely to get a counterpart of Gromov-Milman Theorem for convex sets (that the convex Poincar\'e inequalities \eqref{Poincare convex -} or \eqref{Poincare convex +} imply the dimension free concentration property $\mathbf{CCI}_2^\infty(\alpha)$) with an exponential profile, we need to add the structural assumption that the underlying metric space is of Busemann type.

Recall that $(\X,d)$ is said to be a \emph{Busemann's space} if the distance  $d:\X^2\to\R^+$ is a convex function (of both variables). 
Banach spaces are obvious examples of such spaces. Another important class of Busemann's spaces are complete connected Riemannian manifolds with \emph{non positive} sectional curvature. We refer to \cite{BBI01}, \cite{BH99}, \cite{P05} for more informations on the topic and a proof of this statement.

Let us mention some elementary properties of Busemann's spaces: 
\begin{itemize}
\item[(B1)] if $\X$ is a Busemann's space, any two points of $\X$ are joined by a unique constant speed geodesic; 
\item[(B2)] for any convex subset $A\subset \X$, the function $\X \ni x \mapsto d(x,A)$ is convex;
\item[(B3)] if $(\X,d)$ is a Busemann's space, then $(\X^n,d_2)$ is a Busemann's space.
\end{itemize}

We are now in position to state our second main theorem (a counterpart of Gromov-Milman's Theorem \ref{Gromov Milman}). % The second result of this part is the following.

\begin{thm}\label{Gromov Milman convex} 
Let $(\X,d)$ be a geodesic space and $\mu$ be a probability measure.
Assume one of the following hypothesis:
either
\begin{itemize}
\item[(a)] 
$\X$ is a Busemann's space %satisfying Hypothesis $(H)$ 
and $\mu$ satisfies the convex Poincar\'e inequality \eqref{Poincare convex +} with a constant $\lambda>0$.
\end{itemize}
or
\begin{itemize}
\item[(b)] 
$(\X,\|\cdot\|)$  is a Banach space (with $\|x-y\|=d(x,y)$, $x,y \in \X$) and 
$\mu$ satisfies the convex Poincar\'e inequality \eqref{Poincare convex -} with a constant $\lambda>0$;
\end{itemize}
Then $\mu$ satisfies the dimension free convex concentration property $\mathbf{CCI}_2^\infty(\alpha)$ with the profile 
\[
\alpha(r)=b\exp(-a\sqrt{\lambda}r),\qquad r\geq0,
\]
where $a,b$ are universal constants.
\end{thm}

%
%\begin{thm}\label{Gromov Milman convex} 
%Let $(\X,d)$ be a Busemann geodesic space. %satisfying \emph{one} of the following hypothesis:
%%\begin{itemize}
%%\item[(1)] $(X,d)$ is a complete CAT(0) space;
%%\item[(2)] $(X,d)$ is a Busemann's space satisfying Hypothesis $\mathbf{(H)}$.
%%\end{itemize}
%Assume furthermore that $\mu$ satisfies the Poincar\'e inequality \eqref{Poincare+} restricted to the class of convex functions with a constant $\lambda>0$. Then it satisfies the dimension free convex concentration property $\mathbf{CCI}_2^\infty(\alpha)$ with the profile 
%\[
%\alpha(r)=b\exp(-a\sqrt{\lambda}r),\qquad r\geq0,
%\]
%where $a,b$ are universal constants.
%\end{thm}

%Observe that Theorem \ref{Gromov Milman convex banach} and Theorem \ref{Gromov Milman convex} are dealing with different type of Poincar\'e inequalities.

As a direct corollary of Theorem \ref{Gromov Milman convex} and Theorem \ref{main result convex}, when $\X$ is a  Banach space, we conclude that, similarly to the general case, the set of probability measures $\mu$ satisfying  the dimension free convex concentration property $\mathbf{CCI}_2^\infty(\alpha)$ with exponential profile coincides  with the set of probability measures verifying the convex Poincar\'e inequality \eqref{Poincare convex -}.

The remaining of the section is dedicated to the proof of Theorem  \ref{Gromov Milman convex}. There is a technical obstacle for applying Herbst argument as in Proposition \ref{Herbst}. Namely, since the class of convex functions is not stable under truncation from above (if $f$ is convex and $a \in \R$, $\min(f; a)$ is in general not convex) it is delicate to deal safely with $\int e^{sf}\,d\mu$, $s\geq 0$. Although a method based on $p$-th moments could perhaps be considered in replacement,
%\footnote{Note that the convex Poincar\'e inequalities \eqref{Poincare convex -} and \eqref{Poincare convex +} imply by an immediate recurrence the finiteness of all $p$-th moments of a non negative, convex and Lipschitz function.}
we use instead an argument based on Corollary \ref{Cor Tal 2}: we first show that under \eqref{Poincare convex -} and \eqref{Poincare convex +}, $\mu$ verifies the dimension free convex concentration property $\mathbf{CCI}_2^\infty(\alpha)$ with some polynomial concentration profile and then we upgrade it into an exponential one using Corollary \ref{Cor Tal 2}.

\begin{proof}[Proof of Theorem \ref{Gromov Milman convex}] 
We start as in the proof of Proposition \ref{Gromov Milman} by noticing that the convex Poincar\'e inequalities \eqref{Poincare convex -} and \eqref{Poincare convex +} tensorize properly.  
Therefore, if $\mu$ verifies one of the convex Poincar\'e inequalities, then for all $n\in \N^*$, 
\[
\lambda\mathrm{Var}_{\mu^n}(f)\leq \int \sum_{i=1}^n|\nabla_i^{+/-}f|^2\,d\mu^n,
\]
for all Lipschitz and convex function $f:\X^n\to\R$. In particular, if $f$ is $1$-Lipschitz and such that $\sum_{i=1}^n|\nabla_i^{+/-}f|^2\leq 1$, then $\mathrm{Var}_{\mu^n}(f)\leq 1/\lambda$.
Applying Jensen's inequality, we see that
\[
1/\lambda\geq \mathrm{Var}_{\mu^n}(f)=\inf_{a\in \R}\int (f-a)^2\,d\mu^n\geq \left(\inf_{a\in \R}\int |f-a|\,d\mu^n\right)^2=\left(\int |f-m(f)|\,d\mu^n\right)^2
\]
Thus, it follows immediately from Markov's inequality that 
\begin{equation}\label{dev convex}
\mu^n\left(f>m(f)+r\right)\leq \frac{1}{\sqrt{\lambda}r},\qquad \forall r>0.
\end{equation}

Let us first assume that Assumption $(a)$ holds.
Observe that the function $\X^n\to\R:x \mapsto f_A(x):=d_2(x,A)$
is convex and $1$-Lipschitz when $A\subset\X^n$ is a convex set (this follows from properties $(B_2)$ and $(B_3)$ above). Moreover,
using Point $(i)$ of Lemma \ref{approx} and arguing as in the proof of Theorem \ref{Gromov Milman}, we see that $f_{A,\varepsilon}:=\sqrt{\varepsilon+f_A^2}$ satisfies the condition $\sum_{i=1}^n|\nabla_i^{+}f_A|^2\leq 1$. Since the function $u\mapsto \sqrt{\varepsilon+u^2}$ is convex and increasing, the function 
$f_{A,\varepsilon}$ is itself convex. Applying \eqref{dev convex} to $f_{A,\varepsilon}$ and letting $\varepsilon\to 0$, we get
\begin{equation}\label{conc conv elem}
\mu^n(A_{r,2})\geq 1-\frac{1}{\sqrt{\lambda}r},
\end{equation}
for all convex set $A.$ 

In particular, taking $r_o=4/\sqrt{\lambda}$ we are in the framework of Corollary \ref{Cor Tal 2}, with $a_o=1/4$ and $\mathcal{A}_n$ being the class of convex sets of $\X^n.$ (Note that Assumptions $(i)$ and $(ii)$ are well verified by $\mathcal{A}_n$: a product of convex sets is always convex and properties $(B_2)$ and $(B_3)$ above show that the enlargement of a convex set remains convex). We thus conclude that there are universal constants $0<\gamma<1$ and $1/2\leq c<1$ such that, for all $n\in\N^*$, for all convex set $A$ such that $\mu^n(A)\geq c$
\[
\mu^n(A_{r,2})\geq 1-be^{-a\sqrt{\lambda}r},\qquad \forall r\geq 0,
\]
with $b=(1-c)/\gamma$ and $a=-\log(\gamma)/4$. 
Now, if $\mu^n(A)\geq 1/2$, then applying \eqref{conc conv elem} we see that $\mu^n(A_{r_1,2})\geq c$ for $r_1=\frac{1}{\sqrt{\lambda}(1-c)}$. We easily conclude from this that $\mu$ verifies $\mathbf{CCI}_2^\infty(\alpha)$ with $\alpha(r)=b'e^{-a\sqrt{\lambda}r}$, for some other universal constant $b'$.

Assume now that Assumption $(b)$ holds. Let $A\subset \X^n$ be a convex subset. The function $f_A$ defined above is $1$-Lipschitz and convex. Therefore, according to Point $(i)$ and $(iii)$ Lemma \ref{approx}, the function $R_\varepsilon f_A$, where $R_\varepsilon$ is the operator defined in \eqref{defRt}, is $1$ Lipschitz, \emph{convex} and satisfies $\sum_{i=1}^n|\nabla_i^{-}R_\varepsilon f_A|^2\leq 1$. Applying \eqref{dev convex}, we conclude that
\[
\mu^n\left(R_\varepsilon f_A>m(R_\varepsilon f_A)+r\right)\leq \frac{1}{\sqrt{\lambda}r},\qquad \forall r>0.
\]
Moreover, according to Point $(ii)$ of Lemma \ref{approx}, it holds $f_A-\sqrt{\varepsilon}\leq R_\varepsilon f_A\leq f_A$. Inserting this into the deviation inequality above and letting $\varepsilon\to 0$ yields to the conclusion that \eqref{conc conv elem} holds for all convex $A.$ The rest of the proof is identical.\end{proof}


\begin{thebibliography}{10}

\bibitem{AS94}
S.~Aida and D.~Stroock.
\newblock Moment estimates derived from {P}oincar\'e and logarithmic {S}obolev
  inequalities.
\newblock {\em Math. Res. Lett.}, 1:75--86, 1994.

\bibitem{AGS12}
L.~{Ambrosio}, N.~{Gigli}, and G.~{Savar{\'e}}.
\newblock Calculus and heat flow in metric measure spaces and applications to
  spaces with {R}icci bounds from below.
\newblock To appear in Invent. Math., 2013.

\bibitem{ane}
C.~An{\'e}, S.~Blach{\`e}re, D.~Chafa{\"{\i}}, P.~Foug{\`e}res, I.~Gentil,
  F.~Malrieu, C.~Roberto, and G.~Scheffer.
\newblock {\em Sur les in\'egalit\'es de {S}obolev logarithmiques}, volume~10
  of {\em Panoramas et Synth\`eses [Panoramas and Syntheses]}.
\newblock Soci\'et\'e Math\'ematique de France, Paris, 2000.
\newblock With a preface by Dominique Bakry and Michel Ledoux.

\bibitem{BBI01}
D.~Burago, Y.~Burago and I.~Ivanov
\newblock {\em A course in metric geometry}, volume~33.
\newblock American Mathematical Society, Providence, 2001.

\bibitem{Ba12}
Z.~M. Balogh, A.~Engulatov, L.~Hunziker, and O.~E. Maasalo.
\newblock Functional inequalities and {H}amilton--{J}acobi equations in
  geodesic spaces.
\newblock {\em Potential Anal.}, 36(2):317--337, 2012.

\bibitem{BGL01}
S.~G. Bobkov, I.~Gentil, and M.~Ledoux.
\newblock Hypercontractivity of {H}amilton-{J}acobi equations.
\newblock {\em J. Math. Pures Appl. (9)}, 80(7):669--696, 2001.

\bibitem{BG99}
S.~G. Bobkov, F.~G{\"o}tze.
\newblock Discrete isoperimetric and Poincar\'e-type inequalities.
\newblock {\em Probab. Theory Related Fields}, 114(2):245--277, 1999.

\bibitem{BH99} M. Bridson, A. Haefliger, \emph{Metric spaces of non-positive curvature}, vol. 319 of Grundlehren der Mathematischen Wissenschaften, Springer-Verlag, Berlin, 1999.

\bibitem{BH00}
S.~G. Bobkov and C.~Houdr{\'e}.
\newblock Weak dimension-free concentration of measure.
\newblock {\em Bernoulli}, 6(4):621--632, 2000.

\bibitem{BL97}
S.~G. Bobkov and M.~Ledoux.
\newblock Poincar\'e's inequalities and {T}alagrand's concentration phenomenon
  for the exponential distribution.
\newblock {\em Probab. Theory Related Fields}, 107(3):383--400, 1997.

\bibitem{Bor75}
C.~Borell.
\newblock The {B}runn-{M}inkowski inequality in {G}auss space.
\newblock {\em Invent. Math.}, 30(2):207--216, 1975.


\bibitem{Cheeger}
J.~Cheeger.
\newblock A lower bound for the smallest eigenvalue of the {L}aplacian.
\newblock {\em Problems in analysis ({P}apers dedicated to {S}alomon {B}ochner, 1969)},
Princeton Univ. Press,      195--199, 1970.


\bibitem{Evans}
L.~C. Evans.
\newblock {\em Partial differential equations}, volume~19 of {\em Graduate
  Studies in Mathematics}.
\newblock American Mathematical Society, Providence, RI, second edition, 2010.

\bibitem{F71}
W.~Feller.
  \newblock {\em An introduction to probability theory and its applications.
              {V}ol. {II}. },
\newblock {John Wiley \& Sons Inc.},
 {New York},
      { {Second edition},1971}.
    		

\bibitem{FS12}
K.~Funano and T.~Shioya.
\newblock Concentration, {R}icci curvature and laplacian.
\newblock To appear in Geom. Funct. Anal., 2013.

\bibitem{GL13}
N.~Gigli and M.~Ledoux.
\newblock From log Sobolev to Talagrand: a quick proof.
\newblock {\em Discrete Contin. Dyn. Syst.}, 33(5):1927--1935, 2013.

\bibitem{Goz09}
N.~Gozlan.
\newblock A characterization of dimension free concentration in terms of
  transport inequalities.
\newblock {\em Ann. Probab.}, 37(6):2480--2498, 2009.

\bibitem{Goz10}
N.~Gozlan.
\newblock Poincar\'e inequalities and dimension free concentration of measure.
\newblock {\em Ann. Inst. Henri Poincar\'e Probab. Stat.}, 46(3):703--739, 2010.

\bibitem{GL10}
N.~Gozlan and C.~L{\'e}onard.
\newblock Transport inequalities. {A} survey.
\newblock {\em Markov Process. Related Fields}, 16(4):635--736, 2010.

\bibitem{GRS11}
N.~Gozlan, C.~Roberto, and P-M Samson.
\newblock From concentration to logarithmic {S}obolev and {P}oincar\'e
  inequalities.
\newblock {\em J. Funct. Anal.}, 260(5):1491--1522, 2011.

\bibitem{GRS12}
N.~Gozlan, C.~Roberto, and P.M. Samson.
\newblock Characterization of {T}alagrand's transport-entropy inequalities on
  metric spaces.
\newblock To appear in Ann. Probab., 2012.

\bibitem{GRS12bis}
N.~Gozlan, C.~Roberto, and P.M. Samson.
\newblock Hamilton-{J}acobi equations on metric spaces and transport-entropy
  inequalities.
\newblock To appear in Rev. Mat. Iberoam., 2012.

\bibitem{Gromov-book}
M.~Gromov.
\newblock {\em Metric structures for {R}iemannian and non-{R}iemannian spaces},
  volume 152 of {\em Progress in Mathematics}.
\newblock Birkh\"auser Boston Inc., Boston, MA, 1999.
\newblock Based on the 1981 French original [ MR0682063 (85e:53051)], With
  appendices by M. Katz, P. Pansu and S. Semmes, Translated from the French by
  Sean Michael Bates.

\bibitem{GM83}
M.~Gromov and V.~D. Milman.
\newblock A topological application of the isoperimetric inequality.
\newblock {\em Amer. J. Math.}, 105(4):843--854, 1983.

\bibitem{KLS95}
R.~Kannan, L.~Lov{\'a}sz, and M.~Simonovits.
\newblock Isoperimetric problems for convex bodies and a localization lemma.
\newblock {\em Discrete Comput. Geom.}, 13(3-4):541--559, 1995.

\bibitem{Ledoux-book}
M.~Ledoux.
\newblock {\em The concentration of measure phenomenon}, volume~89 of {\em
  Mathematical Surveys and Monographs}.
\newblock American Mathematical Society, Providence, RI, 2001.

\bibitem{LV07}
J.~Lott and C.~Villani.
\newblock Hamilton-{J}acobi semigroup on length spaces and applications.
\newblock {\em J. Math. Pures Appl. (9)}, 88(3):219--229, 2007.

\bibitem{M86}
K.~Marton.
\newblock A simple proof of the blowing-up lemma.
\newblock {\em IEEE Trans. Inform. Theory}, 32(3):445--446, 1986.

\bibitem{Ma91}
B.~Maurey.
\newblock Some deviation inequalities.
\newblock {\em Geom. Funct. Anal.}, 1(2):188--197, 1991.


\bibitem{Mazya}
V.G.~Maz'ja
\newblock Sobolev spaces.
\newblock {\em Springer Series in Soviet Mathematics.}, Springer-Verlag, 1985.


\bibitem{Mil09a}
E.~Milman.
\newblock On the role of convexity in isoperimetry, spectral gap and
  concentration.
\newblock {\em Invent. Math.}, 177(1):1--43, 2009.

\bibitem{Mil10}
E.~Milman.
\newblock Isoperimetric and concentration inequalities: equivalence under
  curvature lower bound.
\newblock {\em Duke Math. J.}, 154(2):207--239, 2010.

\bibitem{Mil11}
E.~Milman.
\newblock Isoperimetric bounds on convex manifolds.
\newblock In {\em Concentration, functional inequalities and isoperimetry},
  volume 545 of {\em Contemp. Math.}, pages 195--208. Amer. Math. Soc.,
  Providence, RI, 2011.

\bibitem{OV00}
F.~Otto and C.~Villani.
\newblock Generalization of an inequality by {T}alagrand and links with the
  logarithmic {S}obolev inequality.
\newblock {\em J. Funct. Anal.}, 173(2):361--400, 2000.

\bibitem{P05}
   A.~Papadopoulos,
    \newblock {\em Metric spaces, convexity and nonpositive curvature},
    \newblock  {\em IRMA Lectures in Mathematics and Theoretical Physics,
   {6},
EMS, 2005}.
   
\bibitem{R91}
     {W.~Rudin},
     \newblock {\em Functional analysis},
     \newblock {\em International Series in Pure and Applied Mathematics, McGraw-Hill Inc., New York, 1991.}
  

\bibitem{Sch98}
M.~Schmuckenschl{\"a}ger.
\newblock Martingales, {P}oincar\'e type inequalities, and deviation
  inequalities.
\newblock {\em J. Funct. Anal.}, 155(2):303--323, 1998.

\bibitem{SC74}
V.~N. Sudakov and B.~S. Cirel$'$son.
\newblock Extremal properties of half-spaces for spherically invariant
  measures.
\newblock {\em Zap. Nau\v cn. Sem. Leningrad. Otdel. Mat. Inst. Steklov.
  (LOMI)}, 41:14--24, 165, 1974.
\newblock Problems in the theory of probability distributions, II.

\bibitem{Tal91}
M.~Talagrand.
\newblock A new isoperimetric inequality and the concentration of measure
  phenomenon.
\newblock In {\em Geometric aspects of functional analysis (1989--90)}, volume
  1469 of {\em Lecture Notes in Math.}, pages 94--124. Springer, Berlin, 1991.

\bibitem{Tal95}
M.~Talagrand.
\newblock Concentration of measure and isoperimetric inequalities in product
  spaces.
\newblock {\em Publications Math\'e\-matiques de l'I.H.E.S.}, 81:73--203, 1995.

\bibitem{Tal96}
M.~Talagrand.
\newblock Transportation cost for {G}aussian and other product measures.
\newblock {\em Geom. Funct. Anal.}, 6(3):587--600, 1996.

\bibitem{vdv98}
A.W.~Van der Vaart.
\newblock {\em Asymptotic statistics}, volume 3 of {\em
  Cambridge Series in Statistical and Probabilistic Mathematics}.
\newblock Cambridge University Press, Cambridge, 1998.

\bibitem{Vi09}
C.~Villani.
\newblock {\em Optimal transport: {O}ld and {N}ew}, volume 338 of {\em
  Grundlehren der Mathematischen Wissenschaften [Fundamental Principles of
  Mathematical Sciences]}.
\newblock Springer-Verlag, Berlin, 2009.

\end{thebibliography}
\end{document}